\documentclass[10pt]{amsart}

\usepackage{tikz-cd}

\usepackage[english]{babel}

\usepackage[letterpaper,top=2cm,bottom=2cm,left=3cm,right=3cm,marginparwidth=1.75cm]{geometry}

\usepackage{amsmath}
\usepackage{graphicx}
\usepackage{hyperref}
\usepackage{mathtools}
\usepackage{amsmath,arydshln,multirow}
\usepackage[alphabetic]{amsrefs}
\usepackage{arydshln}
\usepackage{cases}
\usepackage{amsmath}
\usepackage{amsfonts}
\usepackage{bm}
\usepackage{arydshln}
\usepackage{amsfonts,amsmath,amssymb,amscd,bbm,amsthm,mathrsfs,dsfont}
\usepackage{mathrsfs}
\usepackage{pb-diagram}
\usepackage{amssymb}
\usepackage[all,cmtip]{xy}

\usepackage{tikz}

\newtheorem{theorem}{Theorem}[section]

\newtheorem{proposition}[theorem]{Proposition}
\newtheorem{lemma}[theorem]{Lemma}

\theoremstyle{definition}
\newtheorem{definition}[theorem]{Definition}
\newtheorem{example}[theorem]{Example}
\newtheorem{proposition-definition}[theorem]{Proposition-Definition}
\newtheorem{definition-theorem}[theorem]{Definition-Theorem}

\newtheorem{corollary}[theorem]{Corollary}

\theoremstyle{remark}
\newtheorem{remark}[theorem]{Remark}

\newtheorem{question}[theorem]{Question}
\numberwithin{equation}{section}

\def\xx{\mathbf{x}}

\def\TT{\mathbb{T}}

\def\Acal{\mathcal{A}}

\def\mod{\opname{mod}\nolimits}

\newcommand{\opname}[1]{\operatorname{\mathsf{#1}}}
\newcommand{\Gr}{\opname{Gr}\nolimits}

\newcommand{\End}{\opname{End}}

\newcommand{\Ext}{\opname{Ext}}

\newcommand{\Fac}{\opname{Fac}}
\newcommand{\Hom}{\opname{Hom}}
\newcommand{\dom}{\opname{dom}}
\newcommand{\supp}{\opname{supp}}

\newcommand{\Sub}{\opname{Sub}}
\newcommand{\add}{\opname{add}\nolimits}

\newcommand{\confer}{\emph{cf.}\ }

\begin{document}

\title[Newton polytopes in cluster algebras and $\tau$-tilting theory]{Newton polytopes in cluster algebras and $\tau$-tilting theory}


\author{Peigen Cao}
\address{School of Mathematical Sciences, University of Science and Technology of China, Hefei, 230026, People's Republic of China
}
\email{peigencao@126.com}


\dedicatory{}


\date{}


\begin{abstract}
We prove that the cluster monomials in non-initial cluster variables are uniquely determined by the Newton polytopes of their $F$-polynomials for  skew-symmetrizable cluster algebras. Accordingly, we prove that the $\tau$-rigid modules and the left finite multi-semibricks in $\tau$-tilting theory are uniquely determined by the Newton polytopes of these modules. The key tools used in the proofs are the left Bongartz completion, $F$-invariant and partial $F$-invariant in the context of cluster algebras and $\tau$-tilting theory.
\end{abstract}

\maketitle


\tableofcontents

\section{Introduction}
Cluster algebras are a class of commutative algebras equipped with an extra combinatorial structure introduced by Fomin and Zelevinsky \cite{fz_2002}. Such algebras are  generated by a special set of
generators, called {\em cluster variables}, which are grouped into overlapping subsets of fixed size,
called {\em clusters}. A {\em seed} is a pair consisting of a cluster ${\bf x}=(x_1,\ldots,x_n)$ and a skew-symmetrizable integer matrix $B$.  New seeds can be obtained from a given one by a procedure called {\em mutation}.
The sets of cluster variables and clusters of a cluster algebra are determined by an initial seed $({\bf x}, B)$ and iterative mutations.
A {\em cluster monomial} is a monomial in cluster variables from the same cluster. Cluster monomials are the central objects to study in cluster algebras.

The {\em $g$-vectors} and {\em $F$-polynomials} are introduced in \cite{fomin_zelevinsky_2007} to study the cluster monomials of a cluster algebra in terms of the initial seed $({\bf x}, B)$. More precisely, each cluster monomial $u$ can be written as 
\[u={\bf x}^{{\bf g}_u}\cdot F_u(\hat y_1,\ldots,\hat y_n)\;\;\;\in\mathbb Z[x_1^{\pm 1},\ldots,x_{n}^{\pm 1}],
\]
where ${\bf g}_u\in\mathbb Z^n$ and $F_u\in\mathbb Z[y_1,\ldots,y_n]$ are the $g$-vector and $F$-polynomial of $u$, and  each $\hat y_k$ is a Laurent monomial in $x_1,\ldots,x_n$. In particular, the $g$-vectors of the initial cluster variables $x_1,\ldots,x_n$ are given by the columns of $I_n$ and the $F$-polynomials of $x_1,\ldots,x_n$ are the constant $1$. 

Fomin and Zelevinsky \cite{fomin_zelevinsky_2007} conjectured that different cluster monomials have different $g$-vectors, which has been confirmed by Derksen, Weyman, and Zelevinsky \cite{DWZ10} for skew-symmetric cluster algebras and by 
 Gross,  Hacking,  Keel, and Kontsevich for skew-symmetrizable cluster algebras.

$\tau$-tilting theory was introduced by Adachi, Iyama, and Reiten \cite{air_2014}, which completes the classic tilting theory from the viewpoint of mutations. Various of the fundamental concepts in
cluster algebras, such as clusters/seeds, mutations, $g$-vectors, $F$-polynomials etc., were generalized to $\tau$-tilting theory, \confer \cite{air_2014,fei_2019a,CGY-2023,Cao-2023}.   Table \ref{table} summarizes the correspondence between cluster algebras and $\tau$-tilting theory.
\begin{table}[ht!]
\begin{equation*}
\begin{array}{|c|c|}
\hline
 &
\\[-3mm]
\text{Cluster algebras}& \hspace{2mm} \tau\text{-tilting theory}
\\[2mm]
\hline
 &
\\[-3mm]
\text{Seeds}& \hspace{2mm} \text{Basic } \tau\text{-tilting pairs}
\\[2mm]
\hline
 &
\\[-2mm]
\text{Green/ red mutations of seeds}& \hspace{2mm} \text{Right/ left mutations of } \tau\text{-tilting pairs}
\\[2mm]
\hline
 &
\\[-2mm]
\text{Initial cluster variables } & \hspace{2mm} (0,P_1),\ldots,(0,P_n)
\\[2mm]
\hline
 &
\\[-3mm]
\text{Non-initial  cluster variables}& \hspace{2mm} \text{Indecomposable }\tau\text{-rigid modules}
\\[2mm]
\hline
 &
\\[-3mm]
\text{Cluster monomials}& \hspace{2mm} \tau\text{-rigid pairs}
\\[2mm]
\hline
 &
\\[-3mm]
g\text{-vectors, }F\text{-polynomials}& \hspace{2mm} g\text{-vectors, }F\text{-polynomials}
\\[2mm]
\hline
&
\\[-3mm]
F\text{-invariant for cluster monomials}& \hspace{2mm} F\text{-invariant and $E$-invariant for }\tau\text{-rigid pairs}
\\[2mm]
\hline
 &
\\[-3mm]
\text{Left/ right Bongartz completion}& \hspace{2mm} \text{Left/ right Bongartz completion}
\\[2mm]
\hline
 &
\\[-3mm]
\text{Dominant sets for seeds}& \hspace{2mm} \text{Torsion classes for $\tau$-tilting pairs}
\\[2mm]
\hline
\end{array}
\end{equation*}
\caption{Cluster algebras vs. $\tau$-tilting theory \label{table}}
\end{table}

It is known that $\tau$-rigid pairs in $\tau$-tilting theory play the role of cluster monomials in cluster algebras. Adachi, Iyama, and Reiten \cite{air_2014} proved that different $\tau$-rigid pairs have different $g$-vectors, which is analogous to the result that different cluster monomials have different $g$-vectors in cluster algebras.

Before discussing the motivation of this paper, we recall the definition of Newton polytopes for polynomials and modules, which will be used throughout. Throughout, let $A$ be a finite dimensional basic algebra over an algebraically closed field $\mathbf{k}$.

\begin{definition}[Newton polytopes of polynomials and modules] \label{def:polytope}
(i) Let $F({\bf y})=\sum_{{\bf v}\in\mathbb N^n}c_{\bf v}{\bf y}^{\bf v}\in\mathbb Z[y_1,\ldots,y_n]$ be a non-zero polynomial. Its {\em Newton polytope} $\mathsf P(F)$ is defined to be the convex hull of the finite set $\{{\bf v}\in\mathbb N^n\mid c_{\bf v}\neq 0\}$.

(ii) The {\em Newton polytope} $\mathsf P(M)$ of a module $M\in\mod A$ is defined to be the convex hull of the dimension vectors of the quotient modules of $M$.
\end{definition}

Recently, there has been growing interest in the study of $F$-polynomials. For example, Jiarui Fei \cite{fei_2019a,fei_2019b} studied the tropical $F$-polynomials of modules. In particular, he proved that  the Newton polytope of the $F$-polynomial $F_M\in\mathbb Z[y_1,\ldots,y_n]$ of a module $M\in\mod A$ coincides with the Newton polytope of $M$.
The $F$-invariant in cluster algebras is introduced in \cite{Cao-2023} as a generalization of the $E$-invariant \cite{DWZ10} in additive cluster categorification and the $\mathfrak{d}$-invariant \cite{kkko-2018,kkop-2020} in monoidal cluster categorification. It turns out that the $F$-invariant is related to the tropical $F$-polynomials.

Since the $g$-vectors can determine the cluster monomials and $\tau$-rigid pairs, it is natural to ask {\em whether the $F$-polynomials can determine the cluster monomials and $\tau$-rigid pairs?} 

In order to consider this question, it is natural to exclude the initial cluster variables $x_1,\ldots,x_n$ and correspondingly the initial indecomposable $\tau$-rigid pairs $(0,P_1),\ldots,(0,P_n)$, because the $F$-polynomials of these initial objects are $1$. So we only need to consider cluster monomials in non-initial cluster variables and $\tau$-rigid modules.

In \cite{ckq_2024}, Keller, Qin, and the author proved that the cluster monomials in non-initial cluster variables are uniquely determined by their $F$-polynomials. The proof there depends on a notion of valuation pairing on (upper) cluster algebras. Typically, the same proof does not work in $\tau$-tilting theory.

Recently, it was proved in \cite{cao-2025b} that the {\em indecomposable} $\tau$-rigid modules and the left finite bricks in $\mod A$ are uniquely determined by their Newton polytopes, which is equivalent to saying that such modules are uniquely determined by the Newton polytopes of their $F$-polynomials, thanks to Fei's result \cite{fei_2019a}. The proof in \cite{cao-2025b} is based on two key results: (i) all functorially finite torsion classes are semistable torsion classes \cite{asai-2021}; and (ii) the brick-$\tau$-rigid correspondence in \cite{DIJ-17}, which is a bijection between the indecomposable $\tau$-rigid modules and the left finite bricks. In the general case, it is proved in \cite{cao-2025b} that if two rigid modules $U$ and $V$ (not necessarily indecomposable) have the same Newton polytope, then $U\oplus V$ remains $\tau$-rigid (see \cite[Theorem 3.2]{cao-2025b}  or Lemma \ref{lem:sum-rigid} in this paper). However, it is not clear whether $U$ and $V$ are isomorphic or not. The main motivation for this paper is the following question.

\begin{question}
{\em Whether the Newton polytopes of $F$-polynomials determine the cluster monomials in non-initial cluster variables and the $\tau$-rigid modules?}
\end{question}

   Recently, the Newton polytopes of $F$-polynomials and modules were extensively studied in the literature, {\em cf.} \cite{BK-2012,BKT-2014,BCDMTY-2024, fei_2019a,fei_2019b,AHIKM-2022,Asai-Iyama-2020,li-pan-2022,cao-2025b}. In particular, the Newton polytopes of modules over preprojective algebras are used to study MV polytopes in \cite{BK-2012, BKT-2014}, which parametrize Lusztig’s canonical basis, \confer \cite{Kamnitzer-2010}.

Before giving the main results of this paper, we first recall the definitions of $\tau$-rigid modules and left finite multi-semibricks.
\begin{definition}[$\tau$-rigid module, multi-semibrick and left finite module]
(i) A module $M\in\mod A$ is called {\em $\tau$-rigid}, if $\Hom_A(M,\tau M)=0$, where $\tau$ is the Auslander-Reiten translation in $\mod A$.

(ii) A module $M\in \mod A$ is called a {\em multi-semibrick} if
$M$ has a decomposition $M\cong \oplus_{i=1}^rC_i^{a_i}$ such that each $C_i\in\mod A$ is a brick and $\Hom_A(C_i,C_j)=0$ for any $i\neq j$.

(iii) A module $M\in\mod A$ is called {\em left finite}, if the smallest torsion class containing $M$ is a functorially finite torsion class in $\mod A$.
\end{definition}

The main results in this paper are as follows.

\begin{theorem}\label{A-thm}
Let $\mathcal A$ be a skew-symmetrizable cluster algebra with initial seed $({\bf x}_{t_0}, B_{t_0})$. Let $u$ and $v$ be two cluster monomials in non-initial cluster variables.
 If the two $F$-polynomials $F_u^{t_0}$ and $F_{v}^{t_0}$ have the same Newton polytope, then $u=v$.   
\end{theorem}

\begin{theorem}\label{B-thm}
Let $A$ be a finite dimensional basic algebra over an algebraically closed field $\mathbf{k}$. Let $U$ and $V$ be two $\tau$-rigid modules in $\mod A$. If $U$ and $V$ have the same Newton polytope, then $U\cong V$.  
\end{theorem}

The left finite multi-semibricks are the dual counterpart of $\tau$-rigid modules in some sense. So it is natural to extend the result on $\tau$-rigid modules to the left finite multi-semibricks.

\begin{theorem}\label{C-thm}
 Let $U$ and $V$ be two left finite multi-semibricks in $\mod A$. If $U$ and $V$ have the same Newton polytope, then $U\cong V$.
\end{theorem}

The proofs of Theorem \ref{A-thm} and Theorem \ref{B-thm} are quite similar, while the proof of Theorem \ref{C-thm} is slightly different. Let us make some comments about the proofs of Theorem \ref{A-thm} and Theorem \ref{B-thm}.  Both proofs are based on the reduction arguments.
The key points are summarized as follows:
\begin{itemize}
    \item [(i)] Show that the direct sum $U\oplus V$ remains a $\tau$-rigid module; and the product $uv$ remains a cluster monomial.
    \item[(ii)] 
The proof of Theorem \ref{B-thm} is
by reducing the number $|U\oplus V|$ of iso-classes of indecomposable direct summands of $U\oplus V$. The proof of Theorem \ref{A-thm} is by reducing the number $|\supp(uv)|$ of cluster variables in the support set $\supp(uv)$ of $uv$, which is the set of cluster variables appearing in the cluster monomial $uv$.
\item[(iii)] In order to reduce the number $|U\oplus V|$, we consider the left Bongartz completion $\mathcal M=(M,P)$ of $U\oplus V$, which is the basic $\tau$-tilting pair such that $\Fac M=\Fac (U\oplus V)$. Then we show that any left mutation of $\mathcal M=(M,P)$ can be used to construct two new $\tau$-rigid modules $U'\in \add U $ and $V'\in \add V$ satisfying that  $U'$ and $V'$ have the same Newton polytope but $|U'\oplus V'|<|U\oplus V|$. Moreover, $U\cong V$ if and only if $U'\cong V'$.

\item[(iv)] In order to reduce the number $|\supp(uv)|$, we consider the left Bongartz completion $[{\bf x}_s]$ of the partial cluster $\supp(uv)$. Then we show that any red mutation of the seed $({\bf x}_s, B_s)$ can be used to construct two new cluster monomials $u'$ and $v'$ satisfying that $u'$ and $v'$ have the same Newton polytope but $\supp(u'v')\subsetneq \supp(uv)$. Moreover, $u=v$ if and only if $u'=v'$. 
\end{itemize}

Of course, many preparations are required to ensure that each stage goes smoothly. These will be detailed in the main body of the paper. Among these, the {\em $F$-invariant} and {\em partial $F$-invariant} in cluster algebra and $\tau$-tilting theory will play a crucial role in the proofs, because such invariants are directly related to the Newton polytopes. We give a detailed discussion on $F$-invariant of cluster monomials in Section \ref{sec:F-cluster} and  that of decorated modules in Section \ref{sec:F-tau-tilting}.

This paper consists of two closely related parts: cluster algebras and $\tau$-tilting theory. These two parts can be read in any order. Personally, I first proved the main result in cluster algebras and then realized that a similar approach can be applied to the $\tau$-tilting theory.

\begin{remark}
In this paper, the prefixes ``left" and ``right" in the terminology ``left mutation", ``left Bongartz completion", ``right mutation",  and ``right Bongartz completion" indicate the direction of change under a suitable partial order: {\em ``left” denotes a move to a smaller element, while ``right” denotes a move to a larger one.} This convention aligns with the familiar ordering of integers $(\dots,-2,-1,0,1,2,\ldots)$, where moving left yields smaller numbers and moving right yields larger ones.
\end{remark}

 {\bf Acknowledgement.} 
 The author is partially supported the National Key R\&D Program of China (2024YFA1013801).

\section{Preliminaries on polytopes and cluster algebras}
Throughout, we denote by $\langle{-,-}\rangle:\mathbb R^n\times \mathbb R^n\to \mathbb R$ the standard inner product on $\mathbb R^n$ and by ${\bf e}_1,\ldots,{\bf e}_n$ the standard basis of $\mathbb R^n$.

\subsection{Polytopes, Minkowski sum and tropical polynomials}
A {\em polytope} $\mathsf P$ in $\mathbb R^n$ is the convex hull of a finite (non-empty)
subset of $\mathbb R^n$, which is a bounded closed subset in $\mathbb R^n$. For a polytope  $\mathsf P$ in $\mathbb R^n$, its {\em support function} $h_{\mathsf P}:\mathbb R^n\to \mathbb R$ is defined by
\[ h_{\mathsf P}({\bf r}):=\max\{ \langle{{\bf a},{\bf r}}\rangle\mid {\bf a}\in\mathsf P\}.\]
It is known from \cite[Section 4.2]{BK-2012} or \cite[Section 1.7]{Schneider_2013} that the polytope $\mathsf P$ can be recovered from its support function $h_{\mathsf P}$ by 
\[ \mathsf P=\{{\bf a}\in\mathbb R^n\mid \langle{{\bf a},{\bf r}}\rangle\leq h_{\mathsf P}({\bf r}),\;\forall {\bf r}\in\mathbb R^n\}.\]

Let $\mathsf P_1$ and $\mathsf P_2$ be two polytopes in $\mathbb R^n$. The {\em Minkowski sum} of $\mathsf P_1$ and $\mathsf P_2$ is the polytope in $\mathbb R^n$ given by
\[\mathsf P_1+\mathsf P_2:=\{ {\bf a}+{\bf b}\mid {\bf a}\in \mathsf P_1,\;{\bf b}\in\mathsf P_2\}.\]

\begin{theorem}[{\cite[Theorem 1.7.5]{Schneider_2013}}] \label{thm:h-additive}
Let $\mathsf P_1$ and $\mathsf P_2$ be two polytopes in $\mathbb R^n$. Then \[h_{\mathsf P_1+\mathsf P_2}=h_{\mathsf P_1}+h_{\mathsf P_2}.\]
\end{theorem}
As a direct consequence, the following cancellation law holds.
\begin{corollary}\label{cor:cancel-law}
    Let $\mathsf P_1, \mathsf P_2$ and $\mathsf Q$ be three polytopes in $\mathbb R^n$. If $\mathsf P_1+\mathsf Q=\mathsf P_2+\mathsf Q$, then $\mathsf P_1=\mathsf P_2$.
\end{corollary}
\begin{proof}
 Since $\mathsf P_1+\mathsf Q=\mathsf P_2+\mathsf Q$ and by Theorem \ref{thm:h-additive}, we have
 \[ h_{\mathsf P_1}+h_{\mathsf Q}=h_{\mathsf P_1+\mathsf Q}=h_{\mathsf P_2+\mathsf Q}=h_{\mathsf P_2}+h_{\mathsf Q}.
 \]
 Thus $h_{\mathsf P_1}=h_{\mathsf P_2}$. Since a polytope is uniquely determined its support function, we get $\mathsf P_1=\mathsf P_2$.
\end{proof}

In this paper, we mainly focus on the polytopes defined from polynomials and modules of a finite dimensional algebra $A$.

\begin{definition}[Newton polytopes of polynomials and modules] \label{def:polytope}
(i) Let $F({\bf y})=\sum_{{\bf v}\in\mathbb N^n}c_{\bf v}{\bf y}^{\bf v}\in\mathbb Z[y_1,\ldots,y_n]$ be a non-zero polynomial. Its {\em Newton polytope} $\mathsf P(F)$ is defined to be the convex hull of the finite set $\{{\bf v}\in\mathbb N^n\mid c_{\bf v}\neq 0\}$.

(ii) The {\em Newton polytope} $\mathsf P(M)$ of a module $M\in\mod A$ is defined to be the convex hull of the dimension vectors of the quotient modules of $M$.
\end{definition}

\begin{proposition}[{\cite[Chapter 6, Prop. 1.2]{GKZ-1994}}]
\label{pro:GKZ}
 Let $F_1$ and $F_2$ be two non-zero polynomials in $\mathbb Z[y_1,\ldots,y_n]$. Then 
 $$\mathsf P(F_1F_2)=\mathsf P(F_1)+\mathsf P(F_2),$$
 where $\mathsf P(F_k)$ is the Newton polytope of $F_k$ for $k=1,2$.
\end{proposition}

\begin{definition}[Tropical polynomial] \label{def:trop-polynomial}
Let $F({\bf y})=\sum_{{\bf v}\in\mathbb N^n}c_{\bf v}{\bf y}^{\bf v}\in\mathbb Z[y_1,\ldots,y_n]$ be a non-zero polynomial. The {\em tropical polynomial} of $F$ is the map $F[-]:\mathbb R^n\to \mathbb R$ defined by
\[F[{\bf r}]:=\max\{\langle{{\bf v},{\bf r}}\rangle\mid c_{\bf v}\neq 0\}.\]
\end{definition}
 
Tropical polynomials play an important role in defining the $F$-invariant in cluster algebras and $\tau$-tilting theory in Sections \ref{sec:F-cluster}, \ref{sec:F-tau-tilting}. We can see that if ${\bf r}\in\mathbb Z^n$, then $F[{\bf r}]\in \mathbb Z$.
\begin{remark}\label{rmk:cons-1}
We have the following important facts.
\begin{itemize}
\item[(a)] The tropical polynomial $F[-]:\mathbb R^n\to\mathbb R$ is uniquely determined by the Newton polytope $\mathsf P(F)$ of $F$. Actually, it only depends on the vertices of the Newton polytope $\mathsf P(F)$.
    \item [(b)] If ${\bf r}\in\mathbb Z^n$, then $F[{\bf r}]\in \mathbb Z$.
    \item[(c)] If the polynomial $F$ has constant term $1$, then $F[{\bf r}]\in\mathbb Z_{\geq 0}$ for any ${\bf r}\in\mathbb Z^n$.
    \item [(d)] The $F$-polynomials of cluster monomials and modules defined later always have constant term $1$ (see Theorem \ref{thm:GHKK} (iv) and Remark \ref{rmk:F-M-1}).
\end{itemize}
\end{remark}
\begin{example}
  Take  $F=1+y_1+y_1y_2\in\mathbb Z[y_1,y_2]$ and ${\bf r}=\begin{bmatrix}
     -2\\1
 \end{bmatrix}$, then
 \begin{eqnarray}
     F[{\bf r}]=\max\{
     \begin{bmatrix}
         0,0
     \end{bmatrix}\begin{bmatrix}
     -2\\1
 \end{bmatrix}, \begin{bmatrix}
         1,0
     \end{bmatrix}\begin{bmatrix}
     -2\\1
 \end{bmatrix}, \begin{bmatrix}
         1,1
     \end{bmatrix}\begin{bmatrix}
     -2\\1
 \end{bmatrix}
     \}=\max\{0,-2,-1\}=0.
     \nonumber
 \end{eqnarray}
\end{example}

The following result can be checked easily.
\begin{proposition}\label{pro:trop-support}
    Let $F\in\mathbb Z[y_1,\ldots,y_n]$ be a non-zero polynomial and $\mathsf P(F)$ its Newton polytope. Then the tropical polynomial of $F$ and 
    the support function 
    of $\mathsf P(F)$ are the same, i.e.,
 \[ F[{\bf r}]=h_{\mathsf P(F)}({\bf r}),\quad \forall {\bf r}\in\mathbb R^n.\]   
\end{proposition}

\begin{corollary}\label{cor-F-prod}
Let $F_1,F_2\in\mathbb Z[y_1,\ldots,y_n]$ be two non-zero polynomials. Then for any ${\bf r}\in\mathbb R^n$, we have $(F_1F_2)[{\bf r}]=F_1[{\bf r}]+F_2[{\bf r}]$.
\end{corollary}
\begin{proof}
By Proposition \ref{pro:GKZ}, we know that  $\mathsf P(F_1F_2)=\mathsf P(F_1)+\mathsf P(F_2)$. Then by Theorem \ref{thm:h-additive}, we have
\[h_{\mathsf P(F_1F_2)}=h_{\mathsf P(F_1)+\mathsf P(F_2)}=h_{\mathsf P(F_1)}+h_{\mathsf P(F_2)}.\] Then the desired result follows from Proposition \ref{pro:trop-support}.
\end{proof}

\subsection{Cluster algebras}
We first recall Fomin-Zelevinsky's  matrix mutation \cite{fz_2002}. 

\begin{definition}[Matrix mutation]Let $A=(a_{ij})$ be an $m\times n$ integer matrix. For any integer $k$ with $k\leq m$ and $k\leq n$, the  {\em mutation} of $A$
   in direction $k$ is
defined to be the new integer matrix $\mu_k(A)=A'=(a_{ij}')$ given by
\begin{eqnarray}\label{eqn:b-mutation}
a_{ij}^\prime&=&\begin{cases}-a_{ij}, & \text{if}\;i=k\;\text{or}\;j=k,\\
 a_{ij}+[a_{ik}]_+[a_{kj}]_+-[-a_{ik}]_+[-a_{kj}]_+,&\text{otherwise},\end{cases}\nonumber
\end{eqnarray}
where $[a]_+:=\max\{a,0\}$ for any $a\in\mathbb R$.
\end{definition}


Now we fix a positive integer $n$ and denote by $[1,n]:=\{1,2,\ldots,n\}$. An $n\times n$ integer matrix $B$ is said to be {\em skew-symmetrizable}, if there exists a diagonal integer matrix $D=diag(d_1,\ldots,d_n)$ with each $d_i>0$ such that $DB$ is skew-symmetric. Such a diagonal matrix $D$ is called a {\em skew-symmetrizer} of $B$.

\begin{proposition}[{\cite{fz_2002}}]
 (i) For any integer matrix $A=(a_{ij})_{m\times n}$, we have $\mu_k^2(A)=A$.

 (ii) If $B=(b_{ij})_{n\times n}$ is skew-symmetrizable, then $B':=\mu_k(B)$ is still skew-symmetrizable and the two matrices $B, B'$ share the same skew-symmetrizers.
\end{proposition}

 A {\em seed}  in  $\mathbb F:=\mathbb Q(z_1,\ldots,z_n)$ is a pair $({\bf x},B)$, where
\begin{itemize}
    \item ${\bf x}=(x_1,\ldots,x_n)$ is an ordered set of free generators of $\mathbb F$ over $\mathbb Q$;
    \item $B=(b_{ij})$ is an $n\times n$ skew-symmetrizable matrix. 
\end{itemize}

\begin{definition}[Seed mutation]  Let $({\bf x},B)$ be a seed in $\mathbb F$. The {\em mutation}  of  $({\bf x},B)$ in direction $k\in[1,n]$ is the new seed  $({\bf x}', B')=\mu_k({\bf x}, B)$ given by $ B'=\mu_k(B)$ and
\begin{eqnarray}
\label{eqn:x-mutation}
 x_i^\prime=\begin{cases}x_i,&
 i\neq k,\\
 x_k^{-1}\cdot (\prod_{j=1}^nx_j^{[b_{jk}]_+}+\prod_{j=1}^nx_j^{[-b_{jk}]_+}),&i= k.\end{cases}\nonumber
\end{eqnarray}

\end{definition}
It can be checked that  $\mu_k$ is an involution. Let $\mathbb T_n$ denote the $n$-regular tree. We
 label the edges of $\mathbb T_n$ by $1,\ldots, n$ such that the $n$ different edges adjacent to the same vertex of $\mathbb T_n$ receive different labels.
 
\begin{definition}[Cluster pattern]  A {\em cluster pattern}  $\mathcal S_X=\{({\bf x}_t, B_t)\mid t\in \mathbb T_n\}$   is an assignment of a seed $({\bf x}_t,  B_t)$ in $\mathbb F$ to every vertex $t$ of $\mathbb T_n$ such that $({\bf x}_{t'}, B_{t'})=\mu_k({\bf x}_t,  B_t)$  whenever
	\begin{xy}(0,1)*+{t}="A",(10,1)*+{t'}="B",\ar@{-}^k"A";"B" \end{xy} in $\TT_n$. 
\end{definition}

Usually, we fix a vertex $t_0\in\mathbb T_n$ as the rooted vertex of $\mathbb T_n$. The seed of a cluster pattern at the rooted vertex $t_0$ is called an {\em initial seed}. We call
${\bf x}_t$ and $B_t$ the {\em cluster} and {\em exchange matrix} at the vertex $t\in\mathbb T_n$ and write
 ${\bf x}_t=(x_{1;t},\ldots,x_{n;t})$ and $B_t=(b_{ij}^t)$. Elements in clusters are called {\em cluster variables}. 

The {\em cluster algebra} $\mathcal A$ associated to a cluster pattern $\mathcal S_X=\{({\bf x}_t, B_t)\mid t\in \mathbb T_n\}$ is the $\mathbb Z$-subalgebra of $\mathbb F=\mathbb Q(z_1,\ldots,z_n)$ given by
 $$\mathcal A=\mathbb Z[x_{1;t},\ldots,x_{n;t}\mid t\in\mathbb T_n].$$

 A {\em cluster monomial} $u$ of $\mathcal A$ is a monomial in cluster variables from the same cluster, i.e., $$u={\bf x}_t^{\bf h}=x_{1;t}^{h_1}\cdots x_{n;t}^{h_n}$$ for some vertex $t\in\mathbb T_n$ and ${\bf h}=(h_1,\ldots,h_n)^T\in\mathbb N^n$.
\begin{theorem}[Laurent phenomenon and separation formula
\cite{fz_2002,fomin_zelevinsky_2007}]
Let $\mathcal A$ be a cluster algebra with initial seed
$({\bf x}_{t_0}, B_{t_0})$. Then the following statements hold.

\begin{itemize}
    \item [(i)] Any cluster monomial $u$ can be written as a Laurent polynomial in  $\mathbb Z[x_{1;t_0}^{\pm 1},\ldots,x_{n;t_0}^{\pm 1}]$.
    \item [(ii)]  Denote by   $\hat y_{k;t_0}={\bf x}_{t_0}^{B_{t_0}{\bf e}_k}$. The Laurent polynomial in (i) has a canonical expression
\begin{eqnarray}\label{eqn:x-gF}
   u={\bf x}_{t_0}^{{\bf g}_{u}^{t_0}}F_{u}^{t_0}(\hat y_{1;t_0},\ldots,\hat y_{n;t_0})\;\;\;\in\;\; \mathbb Z[x_{1;t_0}^{\pm 1},\ldots,x_{n;t_0}^{\pm 1}],\nonumber
\end{eqnarray}
where $g_{u}^{t_0}\in\mathbb Z^n$ and $F_{u}^{t_0}({\bf y})\in\mathbb Z[y_1,\ldots,y_n]$ are canonically defined from principal cluster algebras \cite[(6.4) \& (3.3)]{fomin_zelevinsky_2007}.
\end{itemize}
\end{theorem}

\begin{definition}[$g$-vector, $G$-matrix and $F$-polynomial] Let $u$ be a cluster monomial of $\mathcal A$ and keep the notations above.
\begin{itemize}
\item [(i)] The integer vector ${\bf g}_{u}^{t_0}\in\mathbb Z^n$ is called the {\em $g$-vector} of $u$ with respect to (the seed at) vertex $t_0$. 
\item [(ii)]The  matrix 
    $G_t^{t_0}=({\bf g}_{x_{1;t}}^{t_0},\ldots,{\bf g}_{x_{n;t}}^{t_0})$ 
    is called the {\em $G$-matrix} of $({\bf x}_t, B_t)$ with respect to vertex $t_0$.
\item [(iii)] The polynomial $F_{u}^{t_0}\in\mathbb Z[y_1,\ldots,y_n]$ is called the {\em $F$-polynomial} of $u$ with respect to  vertex $t_0$.
\end{itemize}
\end{definition}

\begin{example}[Cluster algebra of type $A_2$]
\label{ex:A2}
   Take $B=\begin{bmatrix}
    0&1\\-1&0
\end{bmatrix}$ and ${\bf x}=(x_1,x_2)$. 
It is easy to  check that the  cluster algebra $\mathcal A$ defined by the initial seed $({\bf x}, B)$  has only five (unlabeled) clusters \[\{x_1,x_2\},\;\{x_2,x_3\},\;\{x_3,x_4\},\;\{x_4,x_5\},\;\{x_5,x_1\},\quad \text{where}\]
\[x_3:=   \frac{x_2+1}{x_1}, \;x_4:=   \frac{x_1+x_2+1}{x_1x_2}, \;x_5:=  \frac{x_1+1}{x_2}.\]
The canonical expressions of the three non-initial cluster variables with respect to the initial seed $({\bf x}, B)$ are given as follows:
    \[x_3={x_1^{-1}x_2\cdot(1+\widehat y_1),}\;\;\;  x_4={x_1^{-1}\cdot (1+\widehat y_1+\widehat y_1\widehat y_2),}\;\;\;
    x_5={x_2^{-1}\cdot (1+\widehat y_2)},\]
    where ${\hat y}_1={\bf x}^{B{\bf e}_1}=x_2^{-1}$ and ${\hat y}_2={\bf x}^{B{\bf e}_2}=x_1$.
\end{example}

Let $\mathcal A$ be a cluster algebra whose seed at vertex $t\in\mathbb T_n$ is denoted by $({\bf x}_t,B_t)$. Let $t_0$ and $t$ be two vertices of $\mathbb T_n$ and $\overleftarrow{\mu}$ the mutation sequence corresponding to the unique path from the vertex $t_0$ to $t$ in $\mathbb T_n$.
 We apply the mutation sequence $\overleftarrow{\mu}$ to $\begin{pmatrix} B_{t_0}\\ I_{n}\end{pmatrix}$, then the
 resulting matrix $\overleftarrow{\mu}\begin{pmatrix} B_{t_0}\\ I_{n}\end{pmatrix}$ takes the form $\begin{pmatrix} B_t\\  C_t^{t_0}\end{pmatrix}$ for some $n\times n$ integer matrix $C_{t}^{t_0}$.
 
 \begin{definition}[$C$-matrix and $c$-vector]
 Keep the above notations. We call the $n\times n$ integer matrix $C_t^{t_0}$ the {\em $C$-matrix} of $({\bf x}_t, B_t)$ with respect to vertex $t_0\in\mathbb T_n$, whose columns are called {\em $c$-vectors}.
\end{definition}

\begin{theorem}[{\cite{GHKK18},\cite{NZ12}}] \label{thm:GHKK}
Let $\mathcal A$ be a cluster algebra with initial seed $({\bf x}_{t_0}, B_{t_0})$. The following statements hold.
\begin{itemize}
    \item [(i)] For any vertex $t\in \mathbb{T}_n$, we have $(G_t^{t_0})^TDC_t^{t_0}D^{-1}=I_n$, where $D$ is a skew-symmetrizer for the exchange matrices of $\mathcal A$.
    \item[(ii)]  Each row vector of a $G$-matrix $G_t^{t_0}$ is either non-negative or non-positive.
    \item [(iii)] Each column vector of a $C$-matrix $C_t^{t_0}$ is either non-negative or non-positive.
\item[(iv)] The $F$-polynomial
$F_{u}^{t_0}({\bf y})$ of a cluster monomial $u$ is a polynomial in $\mathbb Z_{\geq 0}[y_1,\ldots,y_n]$ with constant term $1$.
\end{itemize}
\end{theorem}

\begin{definition}[Green and red mutation] 
Let $\mathcal A$ be a cluster algebra with initial seed $({\bf x}_{t_0}, B_{t_0})$. A seed mutation $\mu_k({\bf x}_t, B_t)$ in $\mathcal A$ is called a {\em green mutation}, if the $k$-th column of the $C$-matrix $C_t=C_t^{t_0}$ is a non-negative vector. Otherwise, it is called a {\em red mutation}.
\end{definition}

Denote by $[{\bf x}_t]$ the cluster ${\bf x}_t=(x_{1;t},\cdots,x_{n;t})$ up to permutations, that is, $[\mathbf{x}_t]=\{x_{1;t},\cdots,x_{n;t}\}$.
\begin{proposition}\label{prop:C-positive}
  Let $\mathcal A$ be a cluster algebra initial seed $({\bf x}_{t_0}, B_{t_0})$. If the $C$-matrix $C_t^{t_0}$ is a non-negative matrix, then $[{\bf x}_t]=[{\bf x}_{t_0}]$.
\end{proposition}
\begin{proof}
 Let $\check{\mathcal A}$ be the cluster algebra with initial seed $({\bf z}_{t_0}, B_{t_0}^T)$ and we write $({\bf z}_t, B_t^T)$ for its seed at vertex $t$. We use $\check G_{t_2}^{t_1}$ to denote the $G$-matrix of $({\bf z}_{t_2}, B_{t_2}^T)$ with respect to vertex $t_1$. By \cite[(1.13)]{NZ12}, we have $C_{t}^{t_0}=(\check G_{t_0}^t)^T$. Since $C_t^{t_0}$ is a non-negative matrix, we know that $\check G_{t_0}^t$ is a non-negative matrix. Now we take $({\bf z}_t, B_t^T)$ as the initial seed of $\check{\mathcal A}$. The non-negative $G$-matrix $\check G_{t_0}^t$ corresponds to the positive chamber in the scattering diagram \cite{GHKK18} of $\check{\mathcal A}$. This implies 
 $[{\bf z}_{t_0}]=[{\bf z}_t]$. So $\check G_{t_0}^t$ is a permutation matrix. Thus $C_{t}^{t_0}$ is also a permutation matrix. Therefore, $(C_{t}^{t_0})^{-1}$ is a permutation matrix. In particular, it is a non-negative matrix.
Then by Theorem \ref{thm:GHKK} (i), we see that the $G$-matrix $G_t^{t_0}=D(C_{t}^{t_0})^{-1}D^{-1}$ of $\mathcal A$ is a non-negative matrix. So it corresponds to the positive chamber in the scattering diagram of $\mathcal A$. Thus $[{\bf x}_t]=[{\bf x}_{t_0}]$.
\end{proof}

\subsection{Bongartz completion in cluster algebras} A {\em partial cluster} of $\mathcal A$ is a subset of some cluster of $\mathcal A$. Recall that we denote by $[\mathbf{x}_t]=\{x_{1;t},\cdots,x_{n;t}\}$.  

\begin{definition}[Left and right Bongartz completion, \cite{CGY-2023}]
\label{def:completion}
Let $\mathcal A$ be a cluster algebra and $U$ a partial cluster of $\Acal$.
\begin{itemize}
\item[(i)] A cluster $[{\bf x}_s]$ is called the {\em left Bongartz completion} of $U$ with respect to a vertex $t_0\in\mathbb T_n$ if the following two conditions hold.
\begin{itemize}
 \item [(a)] $U$ is a subset of $[\xx_s]$;
 \item[(b)] The $i$-th column of the $C$-matrix $C_s^{t_0}$ is a non-negative vector for any $i$ such that $x_{i;s}\notin U$.
\end{itemize}
\item[(ii)] A cluster $[{\bf x}_s]$ is called the {\em right Bongartz completion} of $U$ with respect to vertex $t_0\in\mathbb T_n$ if the following two conditions hold.
\begin{itemize}
 \item [(a$'$)] $U$ is a subset of $[\xx_s]$;
 \item[(b$'$)] The $i$-th column of the $C$-matrix $C_s^{t_0}$ is a non-positive vector for any $i$ such that $x_{i;s}\notin U$.
\end{itemize}
\end{itemize}
\end{definition}
\begin{remark}
  The left Bongartz completion and right Bongartz completion  are called the Bongartz completion and Bongartz co-completion in \cite{CGY-2023}. 
\end{remark}

Notice that both the existence and uniqueness of left and right Bongartz completion in cluster algebras are not clear from their own definitions. Let us look at the case  $U=\emptyset$. In this case,  we have the following facts:
\begin{itemize}
    \item A cluster $[{\bf x}_s]$ is the left Bongartz completion of $U=\emptyset$ with respect to vertex $t_0$ if and only if the $C$-matrix $C_s^{t_0}$ is a non-negative matrix, which implies that $[{\bf x}_s]=[{\bf x}_{t_0}]$, by Proposition \ref{prop:C-positive}.
    \item A cluster $[{\bf x}_s]$ is the right Bongartz completion of $U=\emptyset$ with respect to vertex $t_0$ if and only if the $C$-matrix $C_s^{t_0}$ is a non-positive matrix. Notice that such a $C$-matrix exists if and only if the exchange matrix $B_{t_0}$ has a green-to-red sequence in the sense of \cite{Muller16}. In particular, this implies that the right Bongartz completion might not exist in general.
\end{itemize}

\begin{theorem}[{\cite[Theorem 4.15]{CGY-2023}}] \label{thm:CGY-completion}
Let $\mathcal A$ be a cluster algebra with initial seed $({\bf x}_{t_0}, B_{t_0})$.
Then for any partial cluster $U$ of $\mathcal A$, there exists a unique cluster $[{\bf x}_s]$ such that $[{\bf x}_s]$ is the left Bongartz completion of $U$ with respect to vertex $t_0$.
\end{theorem}

\begin{remark}
Note that the right Bongartz completion in cluster algebras might not exist in general. But if it exists, it is unique, by \cite[Corollary 5.5]{CGY-2023}. For this paper, we only need to use the left Bongartz completion, which always exists by Theorem \ref{thm:CGY-completion}. 
\end{remark}

Although the remainder of this subsection is not used in the rest of the paper, it is included to provide an analogue of Proposition \ref{pro-minmax} in $\tau$-tilting theory.

Recall that for a nonzero polynomial
$F({\bf y})=\sum_{{\bf v}\in\mathbb N^n}c_{\bf v}{\bf y}^{\bf v}\in\mathbb Z[y_1,\ldots,y_n]$, its tropical polynomial $F[-]:\mathbb R^n\to \mathbb R$ is defined by
\[F[{\bf r}]:=\max\{\langle{{\bf v},{\bf r}}\rangle\mid c_{\bf v}\neq 0\}.\]
  Denote by $\mathcal X$ the set of  cluster variables of $\mathcal A$. Given a cluster monomial $u={\bf x}_t^{\bf h}=\prod x_{i;t}^{h_i}$ in seed $({\bf x}_t, B_t)$, we define its {\em dominant set} with respect to the initial seed $({\bf x}_{t_0}, B_{t_0})$ as follows:
\[\dom^{t_0}(u)=\{z\in\mathcal X\mid F_z^{t_0}[D{\bf g}_u^{t_0}]=0\}.\]
Since the $F$-polynomials of initial cluster variables are $1$, we see that the initial cluster variables are always contained in $\dom^{t_0}(u)$.

\begin{definition}[Dominant set of a seed]
    Let  $({\bf x}_t, B_t)$ be a seed of $\mathcal A$. The {\em dominant set} $\dom^{t_0}[t]$ of $({\bf x}_t, B_t)$ with respect to the initial seed $({\bf x}_{t_0}, B_{t_0})$ is defined to be the dominant set of the  multiplicity free cluster monomial $u_t=\prod_{i=1}^nx_{i;t}$ with full support,  that is,
    \[
    \dom^{t_0}[t] \coloneqq   \dom^{t_0}(u_t)=\{z\in\mathcal X\mid F_z^{t_0}[D{\bf g}_{u_t}^{t_0}]=0\}.
    \]
\end{definition}
\begin{remark}
    The dominant sets are introduced in \cite{Cao-2023} as a replacement of the torsion classes in $\tau$-tilting theory. The non-initial cluster variables in the dominant set $\dom^{t_0}[t]$ correspond to the indecomposable $\tau$-rigid modules contained in the torsion class $\Fac M=\prescript{\bot}{}({\tau M})\cap P^\bot$ for a $\tau$-tilting pair $(M,P)$ in $\tau$-tilting theory,   \confer \cite[Proposition 7.23]{Cao-2023} or Corollary \ref{cor:dom-hom} in this paper.
\end{remark}

\begin{theorem}[\cite{cao-2026a}]
\label{thm:dom-set}
Let $\mathcal A$ be a cluster algebra with initial seed $({\bf x}_{t_0}, B_{t_0})$ and $U$ a partial cluster of $\mathcal A$. Then the following statements hold.
\begin{itemize}
    \item [(i)]    
    A cluster $[{\bf x}_s]$ is the left Bongartz completion of $U$ with respect to vertex $t_0$ if and only if $U\subseteq [{\bf x}_s]$ and  $\dom^{t_0}[s]\subseteq \dom^{t_0}[t]$ for any cluster $[{\bf x}_t]$ with $U\subseteq [{\bf x}_t]$.
    
   \item [(ii)]  A cluster $[{\bf x}_s]$ is the right Bongartz completion of $U$ with respect to vertex $t_0$ if and only if $U\subseteq [{\bf x}_s]$ and $\dom^{t_0}[t]\subseteq \dom^{t_0}[s]$ for any cluster $[{\bf x}_t]$ with $U\subseteq [{\bf x}_t]$. 
   \item[(iii)] Suppose that $[{\bf x}_{s^-}]$ is the left Bongartz completion and $[{\bf x}_{s^+}]$ is the right Bongartz completion of $U$ with respect to vertex $t_0$. Let $[{\bf x}_t]$ be a cluster of $\mathcal A$. Then $U\subseteq [{\bf x}_t]$ if and only if 
   \[ \dom^{t_0}[s^-]\subseteq \dom^{t_0}[t]\subseteq \dom^{t_0}[s^+].\]
\end{itemize}
\end{theorem}

\section{$\tau$-tilting theory}
\subsection{Decorated modules and $\tau$-tilting pairs}
 We fix a finite dimensional basic algebra $A$ over an algebraically closed field $\mathbf{k}$.  Denote by $\mod A$ the category of finitely generated left $A$-modules, and by $\tau$ the Auslander-Reiten translation in $\mod A$.  The  isomorphism classes of indecomposable projective modules in $\mod A$ are denoted by $P_1,\ldots,P_n$.

Given two modules $M, N\in\mod A$, we denote by
\begin{itemize}
\item $\hom_A(M,N):=\dim_{\mathsf k} \Hom_A(M,N)$.
\item $|M|$ the number of non-isomorphic indecomposable direct summands of $M$.
\item $\add M$  the additive closure of $M$ in $\mod A$.
\item $\Fac M$ the subcategory of $\mod A$ consisting of the quotient modules of the modules in $\add M$.
\item $\Sub M$ the subcategory of $\mod A$ consisting of the submodules of the modules in $\add M$.
\item $\prescript{\bot}{}{M} \coloneqq   \{X\in \mod A  \mid \Hom_{A}(X,M)=0\}$.
\item $M^{\bot} \coloneqq   \{Y\in \mod A  \mid \Hom_{A}(M,Y)=0\}$.

\end{itemize}

A pair $\mathcal M=(M,P)$ of modules in $\mod A$ is called a {\em decorated module} of $A$, if $P$ is a projective $A$-module. The modules $M$ and $P$ are respectively called the {\em positive part} and {\em negative part} of $\mathcal M$.
A decorated $A$-module $\mathcal M=(M,P)$ is called {\em negative}, if $M=0$.

\begin{remark}
Note that the negative parts of decorated modules in our definition are projective $A$-modules, whereas the negative parts of decorated modules used in \cite{DWZ10} are semisimple $A$-modules. Clearly, these two types of decorated $A$-modules are in bijection with each other. 
\end{remark}

Let $\mathcal M=(M,P)$ be a decorated $A$-module.
Let $M=\oplus_{i\in I} M_i^{a_i}$ and $P=\oplus_{j\in J}P_j^{b_j}$ be the indecomposable direct sum decompositions of $M$ and $P$. Then we write $\mathcal M=\oplus_{k\in I\sqcup J}\mathcal M_k$, where 
\[
\mathcal M_k=
\begin{cases}
 (M_k,0),& \text{if }k\in I,\\
 (0,P_k),&\text{if } k\in J.
\end{cases}
\]
A decorated $A$-module $\mathcal M=(M,P)$ is called {\em basic}, if both $M$ and $P$ are basic $A$-module. The {\em direct sum} of two decorated $A$-modules  $\mathcal M=(M,P)$ and $\mathcal N=(N,Q)$ are defined as follows:
\[\mathcal M\oplus \mathcal N:=(M\oplus N,P\oplus Q).\]

A module $M\in\mod A$ is called {\em $\tau$-rigid}, if $\Hom_A(M,\tau M)=0$.

\begin{definition}[$\tau$-rigid pair and $\tau$-tilting pair]
Let $\mathcal M=(M,P)$ be a decorated $A$-module.
\begin{itemize}

\item[(i)]  $\mathcal M=(M,P)$ is called \emph{$\tau$-rigid} if $M$ is $\tau$-rigid and $\Hom_A(P,M)=0$. 
\item[(ii)]  $\mathcal M=(M,P)$ is called \emph{$\tau$-tilting} (resp. {\em almost $\tau$-tilting}) if $\mathcal M=(M,P)$ is $\tau$-rigid and  \[|M|+|P|=|A|\;\;\; (\text{resp.}\;\;  |M|+|P|=|A|-1).\]
\end{itemize}
\end{definition}
We always consider modules, decorated modules up to isomorphism. In a basic $\tau$-tilting pair $(M,P)$, it is known from \cite[Proposition 2.3]{air_2014} that $P$ is uniquely determined by $M$.

\begin{theorem}[\cite{air_2014}*{Theorems 2.12, 2.18}] \label{thm-air-mutation}
Let $\mathcal U=(U,Q)$ be a basic $\tau$-rigid pair. Then the following statments hold.
\begin{itemize}
    \item [(i)] We have $\Fac U\subseteq  \prescript{\bot}{}{(\tau U)}\cap Q^\bot$. The equality holds if and only if $\mathcal U=(U,Q)$ is $\tau$-tilting.
    \item[(ii)] Suppose that $\mathcal U=(U,Q)$ is almost $\tau$-tilting. Then $\Fac U\subsetneq \prescript{\bot}{}{(\tau U)}\cap Q^\bot$ and
there exist exactly two basic
$\tau$-tilting pairs $\mathcal M=(M,P)$ and $\mathcal M'=(M^\prime,P^\prime)$ containing $\mathcal U=(U,Q)$ as a direct summand. Moreover,
$$\{\Fac M,\;\Fac M^\prime\}=\{\Fac U,\;\prescript{\bot}{}{(\tau U)}\cap Q^\bot\}.$$
In particular, either $\Fac M\subsetneq \Fac M^\prime$ or
$\Fac M^\prime\subsetneq\Fac M$ holds.
\end{itemize}
\end{theorem}

\begin{definition}[Left and right mutation]
Keep the notations in Theorem \ref{thm-air-mutation}. The operation $(M,P)\mapsto(M^\prime,P^\prime)$ is called a \emph{mutation} of $(M,P)$. If $\Fac M\subsetneq \Fac M^\prime$ holds, we call $(M^\prime, P^\prime)$ a {\em right mutation} of $(M,P)$. If $\Fac M^\prime \subsetneq \Fac M$ holds, we call $(M^\prime, P^\prime)$ a {\em left mutation} of $(M,P)$.
\end{definition}

\begin{proposition}[{\cite[Theorem 2.35]{air_2014}}] \label{pro-air}
Let $\mathcal M=(M,P)$ and $\mathcal N=(N,Q)$ be two basic $\tau$-tilting pairs with $\Fac N\subsetneqq\Fac M$. Then there exists a left mutation $\mathcal M'=(M^\prime,P^\prime)$ of $\mathcal M=(M,P)$ such that $\Fac N \subseteq\Fac M'\subsetneq \Fac M$.
\end{proposition}

\subsection{$g$-vectors and $F$-polynomials of decorated modules}
Let $M$ be a module in $\mod A$ and let 
\[\bigoplus_{i=1}^nP_i^{b_i}\rightarrow
\bigoplus_{i=1}^nP_i^{a_i}\rightarrow M\rightarrow 0
\]
be the minimal projective presentation of $M$ in $\mod A$. The
vector $$\delta_M \coloneqq   (a_1-b_1,\ldots,a_n-b_n)^T\in\mathbb Z^n$$ is called the {\em $\delta$-vector} of $M$ and the vector ${\bf g}_M \coloneqq   -\delta_M$ is called the {\em $g$-vector} of $M$.

    For a decorated $A$-module $\mathcal M=(M,P)$, we define its $\delta$-vector and $g$-vector as follows:
\[ 
    \delta_{\mathcal M} \coloneqq   \delta_M-\delta_P,\;\;\; {\bf g}_{\mathcal M} \coloneqq   -\delta_{\mathcal M}={\bf g}_{M}-{\bf g}_P.
\]

With this definition, we can see that the $g$-vector ${\bf g}_{(0,P_k)}$ of $(0,P_k)$ is the $k$-th column of $I_n$. This also corresponds to the $g$-vector of the $k$-th initial cluster variable $x_k$ in cluster algebras. 
\begin{remark}
    Note that the $\delta$-vectors defined here coincide with the $g$-vectors used in \cite{air_2014}. For the considerations on the cluster algebras side, the $g$-vectors defined here are the negative of the $\delta$-vectors. 
\end{remark}

\begin{definition}[$F$-polynomial and dual $F$-polynomial]
(i) The {\em $F$-polynomial} $F_M$ of a module $M\in\mod A$ is defined to be
\[ F_M=\sum_{{\bf v}\in\mathbb N^n}\chi(\Gr_{\bf v}(M)){\bf y}^{\bf v}\in\mathbb Z[y_1,\ldots,y_n],
\]
where $\Gr_{\bf v}(M)$ is the {\em quotient module Grassmannian} of $M$ with dimension vector ${\bf v}$ and $\chi$ is the Euler–Poincar\'{e} characteristic.

(ii) The {\em dual $F$-polynomial} $\check F_M$ of a module $M\in\mod A$ is defined to be
\[ \check F_M=\sum_{{\bf v}\in\mathbb N^n}\chi(\check\Gr_{\bf v}(M)){\bf y}^{\bf v}\in\mathbb Z[y_1,\ldots,y_n],
\]
where $\check \Gr_{\bf v}(M)$ is the  {\em submodule Grassmannian} of $M$ with dimension vector ${\bf v}$ and $\chi$ is the Euler–Poincar\'{e} characteristic.

(iii)  The {\em $F$-polynomial} of a decorated $A$-module $\mathcal M=(M,P)$ is defined by $F_{\mathcal M}:= F_M$.
\end{definition}

\begin{remark}\label{rmk:F-M-1}
Since the zero module is a quotient module of $M\in \mod A$,  the polynomial $F_M$ has constant term $1$. Similarly, the dual $F$-polynomial $\check{F}_M$ also has constant term $1$. Thus we have
\[ F_M[{\bf r}], \;\check F_M[{\bf r}]\in\mathbb Z_{\geq 0},\;\;\;\forall {\bf r}\in\mathbb Z^n,
\]
where $F_M[-],\check F_M[-]:\mathbb R^n\to \mathbb R$ are the tropical polynomials defined in Definition \ref{def:trop-polynomial}.
\end{remark}

For a negative decorated $A$-module $\mathcal M=(0,P)$, we clearly have 
$F_{\mathcal M}=1$.

\begin{proposition}[{\cite[Proposition 3.2]{DWZ10}}] \label{pro:DWZ-FmFn}
Let $M$ and $N$ be two modules in $\mod A$. Then \[F_{M\oplus N}=F_M\cdot F_N.\]
\end{proposition}

\begin{proposition}
    [{\cite[Theorem 1.4]{fei_2019a}}] \label{lem:Fei-vertex}
  For each module $M\in\mod A$, we have $\mathsf P(F_M)=\mathsf P(M)$, where $\mathsf P(F_M)$ is 
  the Newton polytope of the $F$-polynomial $F_M$ and $\mathsf P(M)$ is the Newton polytope of module $M$.
\end{proposition}

\begin{corollary}\label{cor:minkowski-sum}
    Let $M$ and $N$ be two modules in $\mod A$. Then $\mathsf P(M\oplus N)=\mathsf P(M)+\mathsf P(N)$.
\end{corollary}
\begin{proof}
By Propositon \ref{pro:DWZ-FmFn} and Propositon \ref{pro:GKZ}, we know that \[\mathsf P(F_{M\oplus N})=\mathsf P(F_{M}F_{N})=\mathsf P(F_M)+\mathsf P(F_N).\] Then the desired result follows from Proposition \ref{lem:Fei-vertex}.
\end{proof}

\begin{example}
    Let $A$ be the path algebra of the quiver $1\rightarrow 2$. We have the exact sequence 
    \[0\rightarrow P_2\rightarrow P_1\rightarrow S_1\rightarrow 0\] in $\mod A=\add (P_1\oplus P_2\oplus S_1)$. There are $5$ basic  $\tau$-tilting pairs in $\mod A$ given as follows:
\[
\xymatrix{
&(0,P_1\oplus P_2)\ar[ld]_{\mu_{(0,P_1)}}\ar[rd]^{\mu_{(0,P_2)}}&\\
(S_1,P_2) \ar[d]_{\mu_{(0,P_2)}}&&(P_2,P_1) \ar[d]^{\mu_{(0,P_1)}}\\
(S_1\oplus P_1,0)\ar[rr]^{\mu_{(S_1,0)}}&&\mathcal (P_1\oplus P_2,0)}
\]
It is easy to check that 
\begin{gather*}
F_{S_1}=1+y_1,\;\;F_{P_1}=1+y_1+y_1y_2,\;\;F_{P_2}=1+y_2,\nonumber\\
   {\bf g}_{S_1}=\begin{bmatrix}
       -1\\ 1
   \end{bmatrix},\;\;
   {\bf g}_{P_1}=\begin{bmatrix}
       -1\\ 0
   \end{bmatrix},\;\;
   {\bf g}_{P_2}=\begin{bmatrix}
       0\\ -1
   \end{bmatrix}.
\end{gather*}
From the viewpoint of categorification of cluster algebras, the indecomposable $\tau$-rigid modules $S_1, P_1, P_2$ correspond to the non-initial cluster variables  $$x_3={x_1^{-1}x_2\cdot(1+\widehat y_1),}\;\;\;  x_4={x_1^{-1}\cdot (1+\widehat y_1+\widehat y_1\widehat y_2),}\;\;\;
    x_5={x_2^{-1}\cdot (1+\widehat y_2)}$$ 
    in Example \ref{ex:A2}.
\end{example}

\subsection{Bongartz completion in $\tau$-tilting theory}
 A \emph{torsion pair} $(\mathcal{T},\mathcal{F})$ in  $\mod A$ is a pair of subcategories of $\mod A$ satisfying that
\begin{itemize}
\item[(i)] $\Hom_{A}(T,F)=0$ for any $T\in \mathcal{T}$ and $F\in \mathcal{F}$;
\item[(ii)] for any object $X\in \mod A$, there exists a short exact sequence
 $$0\rightarrow X_t\rightarrow X\rightarrow X_f\rightarrow0$$ with $X_t\in\mathcal{T}$ and $X_f\in \mathcal{F}$. Thanks to the condition (i), such a sequence is unique up to isomorphisms. This short exact sequence is called the \emph{canonical sequence} of $X$ with respect to $(\mathcal{T},\mathcal{F})$.
\end{itemize}

Notice that in a torsion pair  $(\mathcal T,\mathcal F)$, we always have $\mathcal F=\mathcal T^\bot$ and $\mathcal T= \prescript{\bot}{}{\mathcal F}$.  The subcategory $\mathcal T$ (resp., $\mathcal F$) in a torsion pair $(\mathcal T,\mathcal F)$ is called a \emph{torsion class} (resp., \emph{torsion-free class}) in $\mod A$. 

It is known that a subcategory of $\mod A$ is a torsion class if and only if it is closed under extensions and quotients.
A torsion class $\mathcal T$ is said to be {\em functorially finite}, if there exists a module $M\in\mod A$ such that $\mathcal T=\Fac M$.

A module $M$ in a subcategory $\mathcal C$ of $\mod A$ is said to be {\em Ext-projective} in $\mathcal C$, if 
$\Ext_A^1(M, X)=0$ for any $X\in\mathcal C$. We denote by $\mathcal P(\mathcal C)$ the direct sum of one copy of each of the indecomposable Ext-projective objects in $\mathcal C$ up to isomorphism.

\begin{theorem}[{\cite[Proposition 1.2, Theorem 2.7]{air_2014}}] \label{thm:air-torsion}
The following statements hold.
\begin{itemize}
    \item [(i)] There is a map $\Psi$ 
from $\tau$-rigid pairs to functorially finite torsion classes in $\mod A$
given by $(M,P)\mapsto \Fac M$.

\item[(ii)] The above map $\Psi$ is a bijection if we restrict it to basic $\tau$-tilting pairs.

\item[(iii)] Let $(M,P)$ be a basic $\tau$-tilting pair. Then $M=\mathcal P(\Fac M)$.
\end{itemize}
\end{theorem}

\begin{proposition}[\cite{air_2014}*{Proposition 2.9, Theorem 2.10}]\label{pro-minmax}
 Let $(U,Q)$ be a basic $\tau$-rigid pair in $\mod A$. Then
 \begin{itemize}
 \item[(i)]  $\Fac U$ and  $\prescript{\bot}{}({\tau U})\cap Q^\bot$ are functorially finite torsion classes in $\mod A$.
 \item[(ii)]  $(U,Q)$ is a direct summand of a $\tau$-tilting pair $(M,P)$ if and only if 
 \[\Fac U\subseteq \Fac M\subseteq \prescript{\bot}{}({\tau U})\cap Q^\bot.\]
 \end{itemize}
\end{proposition}

\begin{definition}[Left and right Bongartz completion]\label{def:comp}
Let $(U,Q)$ be a basic $\tau$-rigid pair in $\mod A$ and  $\Psi$ the bijection in Theorem \ref{thm:air-torsion} (ii).
\begin{itemize}
\item[(i)] The {\em left Bongartz completion} (or {\em Bongartz co-completion}) of $(U,Q)$ is defined to be the basic $\tau$-tilting pair $(M^-,P^-)$ such that $\Fac M^-=\Fac U$.

 \item[(ii)] The {\em right Bongartz completion}  (or {\em Bongartz completion}) of $(U,Q)$ is defined to be the basic $\tau$-tilting pair 
 $(M^+,P^+)$ such that $\Fac M^+=\prescript{\bot}{}({\tau U})\cap Q^\bot$.
\end{itemize}
\end{definition}
\begin{remark}\label{rmk:left-completion}
 It is easy to see that the left Bongartz completion $(M^-,P^-)$ of $(U,Q)$
 can be characterized by the following two conditions:
 \begin{itemize}
     \item [(a)] $(U,Q)$ is a direct summand of $(M^-,P^-)$;
     \item[(b)] $\Fac M^-\subseteq \Fac M$ for any basic $\tau$-tilting pair $(M,P)$ such that $(U,Q)$ is a direct summand of $(M,P)$.
 \end{itemize}
 The right Bongart completion $(M^+,P^+)$ of $(U,Q)$ can be characterized in a similar way.
\end{remark}

\section{$F$-invariant and Newton polytopes in cluster algebras}

\subsection{$F$-invariant of cluster monomials}\label{sec:F-cluster}

$F$-invariant in cluster algebras is introduced by the author in \cite{Cao-2023}, which generalizes the $E$-invariant \cite{DWZ10} in additive categorification of cluster algebras and the $\mathfrak{d}$-invariant \cite{kkko-2018,kkop-2020} in monoidal categorification of cluster algebras. 

Recall that for a nonzero polynomial
$F({\bf y})=\sum_{{\bf v}\in\mathbb N^n}c_{\bf v}{\bf y}^{\bf v}\in\mathbb Z[y_1,\ldots,y_n]$, its tropical polynomial $F[-]:\mathbb R^n\to \mathbb R$ is defined by
\[F[{\bf r}]:=\max\{\langle{{\bf v},{\bf r}}\rangle\mid c_{\bf v}\neq 0\}.\]

\begin{definition}[$F$-invariant and partial $F$-invariant] \label{def:F-inv}
Let $\mathcal A$ be a cluster algebra and $D=diag(d_1,\ldots,d_n)$ a fixed skew-symmetrizer for the exchange matrices of $\mathcal A$.
Let $u$ and $v$ be two cluster monomials of $\mathcal A$, and let 
\[
 u={\bf x}_t^{{\bf g}_u^t}F_u^t(\hat y_{1;t},\ldots,\hat y_{n;t}) \;\;\;\;\text{and}\;\;\; v={\bf x}_t^{{\bf g}_{v}^t}F_{v}^t(\hat y_{1;t},\ldots,\hat y_{n;t})
\]
be the canonical expressions of $u$ and $v$ with repect to a vertex $t\in \mathbb T_n$.  
\begin{itemize}
    \item [(i)] 
The integer 
\begin{eqnarray}\label{eqn:def-F-inv}
    (u\mid\mid v)_F=F_u^t[D{\bf g}_{v}^t]+F_{v}^t[D{\bf g}_{u}^t]\nonumber
\end{eqnarray}
is called the {\em $F$-invariant} of  $(u,v)$, which is independent of the choice of $t\in\mathbb T_n$ by Theorem \ref{thm:F-inv} below.
\item[(ii)]  The integer $F_u^t[D{\bf g}_{v}^t]$ is called the {\em partial $F$-invariant} of $(u,v)$ at vertex $t\in\mathbb T_n$.
\end{itemize}
\end{definition}
Since the definition of $F$-invariant depends on the choice of a skew-symmetrizer $D$, {\em we always fix a skew-symmetrizer for the exchange matrices of $\mathcal A$ in this paper.}

\begin{theorem}[{\cite[Theorem 4.10, Proposition 4.30]{Cao-2023}}]
\label{thm:F-inv}
    Let $u$ and $v$ be two cluster monomials of $\mathcal A$ and $D=diag(d_1,\ldots,d_n)$ a fixed skew-symmetrizer for the exchange matrices of $\mathcal A$. Then 
        for any two vertices $t,t'\in\mathbb T_n$, we have 
\begin{eqnarray}\label{eqn:F-inv-mut}
            F_u^t[D{\bf g}_{v}^t]+F_{v}^t[D{\bf g}_u^t]=F_u^{t'}[D{\bf g}_{v}^{t'}]+F_{v}^{t'}[D{\bf g}_u^{t'}].\nonumber
        \end{eqnarray}
        In particular, the $F$-invariant $(u\mid\mid v)_F=F_u^t[D{\bf g}_{v}^t]+F_{v}^t[D{\bf g}_u^t]$ only depends on $u$ and $v$, not on the choice of vertex $t\in\mathbb T_n$.
\end{theorem}

Since the $F$-polynomials of cluster monomials have constant term $1$, we have \[(u\mid\mid v)_F=F_u^t[D{\bf g}_{v}^t]+F_{v}^t[D{\bf g}_{u}^t]\geq 0.\]
\begin{theorem}[{\cite[Theorem 4.19]{Cao-2023}}]
\label{thm:F-compatible}
     Let $u$ and $v$ be two cluster monomials of $\mathcal A$. Then the product $u\cdot v$ remains a cluster monomial of $\mathcal A$ if and only if $(u\mid\mid v)_F=0$.
\end{theorem}

\begin{example}
Let us continue the example of cluster algebra $\mathcal A$  of type $A_2$ in Example \ref{ex:A2}. We know that 
the canonical expressions of $x_3,x_4,x_5$ with respect to  $t_0=({\bf x}, B)$ are given as follows:
\[x_3={x_1^{-1}x_2\cdot(1+\widehat y_1),}\;\;\;  x_4={x_1^{-1}\cdot (1+\widehat y_1+\widehat y_1\widehat y_2),}\;\;\;
    x_5={x_2^{-1}\cdot (1+\widehat y_2)}.\]
  We take $D=I_2$ to be the fixed skew-symmetrizer for the exchange matrices of $\mathcal A$. We have
\begin{align}
    (x_3\mid\mid x_4)_F&=F_{x_3}^{t_0}[D{\bf g}_{x_4}^{t_0}]+F_{x_4}^{t_0}[D{\bf g}_{x_3}^{t_0}]\nonumber\\
    &=(1+y_1)\begin{bmatrix}
        -1\\0
    \end{bmatrix}+(1+y_1+y_1y_2)\begin{bmatrix}
        -1\\
        1
    \end{bmatrix}\nonumber\\
    &=\max\{0,\;-1\}+\max\{0,\;-1,\;0\}\nonumber\\
    &=0,\nonumber\\
    (x_3\mid\mid x_5)_F&=(1+y_1)\begin{bmatrix}
        0\\
        -1
    \end{bmatrix}+(1+y_2)\begin{bmatrix}
        -1\\
        1
    \end{bmatrix}\nonumber\\
    &=\max\{0,\;0\}+\max\{0,\;1\}\nonumber\\
    &=1.\nonumber
\end{align}
Then by Theorem \ref{thm:F-compatible}, we know that $x_3x_4$ is a cluster monomial, while $x_3x_5$ is not.
\end{example}

\begin{proposition}[{\cite[Propositions 4.9, 4.16]{Cao-2023}}] \label{pro:F-additive}
   Let $u=\prod_{i=1}^nx_{i;t}^{h_i}$ and $v$ be two cluster monomials of $\mathcal A$. Then we have $(u\mid\mid v)_F=\sum_{i=1}^nh_i(x_{i;t}\mid\mid v)_F$. Moreover, we have
   \[ F_{v}^{w}[D{\bf g}_u^{w}]=\sum_{i=1}^nh_iF_{v}^w[D{\bf g}_{x_{i;t}}^w]\;\;\;\text{and}\;\;\;F_u^w[D{\bf g}_{v}^w]=\sum_{i=1}^nh_iF_{x_{i;t}}^w[D{\bf g}_{v}^w]
   \]
   for any vertex $w\in\mathbb T_n$.
\end{proposition}

\subsection{Newton polytopes of $F$-polynomials}

\begin{lemma} \label{lem:F-zx=0}
Let $\mathcal A$ be a cluster algebra with initial seed $({\bf x}_{t_0}, B_{t_0})$. If two cluster variables $x$ and $z$ are contained in the same cluster,
then $F_z^{t_0}[D{\bf g}_x^{t_0}]=0=F_x^{t_0}[D{\bf g}_z^{t_0}]$.
\end{lemma}
\begin{proof}
 Since the two cluster variables $x$ and $z$ are contained in the same cluster and by Theorem \ref{thm:F-compatible}, we have $F_x^{t_0}[D{\bf g}_z^{t_0}]+F_z^{t_0}[D{\bf g}_x^{t_0}]=(x\mid\mid z)_F=0$.  Since $F_x^{t_0}[D{\bf g}_z^{t_0}]\geq 0$ and $F_z^{t_0}[D{\bf g}_x^{t_0}]\geq 0$, we get $F_z^{t_0}[D{\bf g}_x^{t_0}]=0=F_x^{t_0}[D{\bf g}_z^{t_0}]$.
\end{proof}

\begin{proposition}\label{prop:exchange-pair}
    Let $\mathcal A$ be a cluster algebra with initial seed $({\bf x}_{t_0}, B_{t_0})$ and $D=diag(d_1,\ldots,d_n)$ a fixed skew-symmetrizer for the exchange matrices of $\mathcal A$. Let $({\bf x}_{t'}, B_{t'})=\mu_k({\bf x}_t, B_t)$ be a mutation in $\mathcal A$. Then $$(x_{k;t}\mid\mid x_{k;t'})_F=d_k.$$ If moreover, the $k$-th column of $C_t^{t_0}$ is non-positive, then we have  
    \[ F_{x_{k;t'}}^{t_0}[D{\bf g}_{x_{k;t}}^{t_0}]=0\quad\text{and}\quad  F_{x_{k;t}}^{t_0}[D{\bf g}_{x_{k;t'}}^{t_0}]=d_k.
    \]
\end{proposition}
\begin{proof}
 Let us use the vertex $t$ to calculate the $F$-invariant $(x_{k;t}\mid\mid x_{k;t'})_F$. Since
 $F_{x_{k;t}}^t=1, F_{x_{k;t'}}^t=1+y_k$ and ${\bf g}_{x_{k;t}}^t={\bf e}_k$,  We have
 \[
  (x_{k;t}\mid\mid x_{k;t'})_F=F_{x_{k;t}}^t[D{\bf g}_{x_{k;t'}}^t]+F_{x_{k;t'}}^t[D{\bf g}_{x_{k;t}}^t]=0+(1+y_k)[D{\bf e}_k]=d_k.
 \]
 On the other hand, if we use the vertex $t_0$ to calculate $(x_{k;t}\mid\mid x_{k;t'})_F$, we have
\[d_k=(x_{k;t}\mid\mid x_{k;t'})_F=F_{x_{k;t}}^{t_0}[D{\bf g}_{x_{k;t'}}^{t_0}]+F_{x_{k;t'}}^{t_0}[D{\bf g}_{x_{k;t}}^{t_0}].\]

Now let us prove that $F_{x_{k;t'}}^{t_0}[D{\bf g}_{x_{k;t}}^{t_0}]=0$ under the assumption that the $k$-th column ${\bf c}_{k;t}$ of $C_t^{t_0}$ is non-positive. In this case, we have $-{\bf c}_{k;t}\in\mathbb Z_{\geq 0}^n$. By the mutation rule \cite[Proposition 5.1]{fomin_zelevinsky_2007} of $F$-polynomials, we have
\[ F_{x_{k;t}}^{t_0}\cdot F_{x_{k;t'}}^{t_0}=\prod_{b_{ik}^t>0}(F_{x_{i;t}}^{t_0})^{b_{ik}^t}+{\bf y}^{-{\bf c}_{k;t}}\cdot \prod_{b_{ik}^t<0}(F_{x_{i;t}}^{t_0})^{-b_{ik}^t}.
\]
By Corollary \ref{cor-F-prod}, we have the tropical version of the above equality:
\begin{eqnarray}\label{eqn:trop-exchange}
    F_{x_{k;t}}^{t_0}[{\bf r}]+F_{x_{k;t'}}^{t_0}[{\bf r}]=\max\{\sum_{b_{ik}^t>0}b_{ik}^tF_{x_{i;t}}^{t_0}[{\bf r}],\;\langle -{\bf c}_{k;t},{\bf r}\rangle+\sum_{b_{ik}^t<0}(-b_{ik}^t)F_{x_{i;t}}^{t_0}[{\bf r}]\}
\end{eqnarray}
for any ${\bf r}\in\mathbb R^n$.  By Lemma \ref{lem:F-zx=0}, we know that $F_{x_{i;t}}^{t_0}[D{\bf g}_{x_{k;t}}^{t_0}]=0$ for any $i\in[1,n]$. Then by taking ${\bf r}=D{\bf g}_{x_{k;t}}^{t_0}$ in \eqref{eqn:trop-exchange}, we get
\[0+ F_{x_{k;t'}}^{t_0}[D{\bf g}_{x_{k;t}}^{t_0}]=\max\{0,\;\langle -{\bf c}_{k;t},D{\bf g}_{x_{k;t}}^{t_0}\rangle +0\}=\max\{0,\;\langle -{\bf c}_{k;t},D{\bf g}_{x_{k;t}}^{t_0}\rangle\}.\]
By Theorem \ref{thm:GHKK} (i), we have $(G_t^{t_0})^TDC_t^{t_0}D^{-1}=I_n$. Thus $\langle -{\bf c}_{k;t},D{\bf g}_{x_{k;t}}^{t_0}\rangle\}=-d_k$. Hence, we have \[F_{x_{k;t'}}^{t_0}[D{\bf g}_{x_{k;t}}^{t_0}]=\max\{0,\;\langle -{\bf c}_{k;t},D{\bf g}_{x_{k;t}}^{t_0}\rangle\}=\max\{0,\;-d_k\}=0.\]
Thus $F_{x_{k;t}}^{t_0}[D{\bf g}_{x_{k;t'}}^{t_0}]=(x_{k;t}\mid\mid x_{k;t'})_F-F_{x_{k;t'}}^{t_0}[D{\bf g}_{x_{k;t}}^{t_0}]=d_k-0=d_k$.
\end{proof}

Recall that for a non-zero polynomial $F\in\mathbb Z[y_1,\ldots, y_n]$, $\mathsf P(F)$ denotes the Newton polytope of $F$.
\begin{lemma}\label{lem:ak=bk}
Let $u=\prod_{i=1}^nx_{i;t}^{a_i}$ and $v=\prod_{i=1}^nx_{i;t}^{b_i}$ be two cluster monomials of $\mathcal A$ in some cluster ${\bf x}_t$. If $\mathsf P(F_u^{t_0})=\mathsf P(F_{v}^{t_0})$, then $a_k=b_k$ whenever the $k$-th column of the $C$-matrix $C_t^{t_0}$ is non-positive.
\end{lemma}

\begin{proof}
Suppose that $k$-th column of the $C$-matrix $C_t^{t_0}$ is non-positive. By Proposition \ref{prop:exchange-pair}, we have \[F_{x_{k;t}}^{t_0}[D{\bf g}_{x_{k;t'}}^{t_0}]=d_k>0.\] For $i\neq k$, we know that 
$\{x_{i;t}, x_{k;t'}\}\subseteq [{\bf x}_{t'}]$. 
Then by Lemma \ref{lem:F-zx=0}, we have
\begin{eqnarray}\label{eqn:Fdg=0}
   F_{x_{i;t}}^{t_0}[D{\bf g}_{x_{k;t'}}^{t_0}]=0, \quad \forall i\neq k. 
\end{eqnarray}
Since $u=\prod_{i=1}^nx_{i;t}^{a_i}$ and $v=\prod_{i=1}^nx_{i;t}^{b_i}$ and
by Proposition \ref{pro:F-additive}, we have
\begin{eqnarray*}
 F_u^{t_0}[D{\bf g}_{x_{k;t'}}^{t_0}]&=&\sum_{i=1}^na_iF_{x_{i;t}}^{t_0}[D{\bf g}_{x_{k;t'}}^{t_0}] \overset{\eqref{eqn:Fdg=0}}{=}a_kF_{x_{k;t}}^{t_0}[D{\bf g}_{x_{k;t'}}^{t_0}]=a_kd_k,\\
 F_{v}^{t_0}[D{\bf g}_{x_{k;t'}}^{t_0}]&=&\sum_{i=1}^nb_iF_{x_{i;t}}^{t_0}[D{\bf g}_{x_{k;t'}}^{t_0}] \overset{\eqref{eqn:Fdg=0}} {=}b_kF_{x_{k;t}}^{t_0}[D{\bf g}_{x_{k;t'}}^{t_0}]=b_kd_k.
\end{eqnarray*}
Since the two $F$-polynomials $F_{u}^{t_0}$ and  $F_{v}^{t_0}$ have the same Newton polytope, we know that $F_u^{t_0}[{\bf r}]=F_{v}^{t_0}[{\bf r}]$ for any vector ${\bf r}\in\mathbb Z^n$. In particular, we have
\[ a_kd_k= F_u^{t_0}[D{\bf g}_{x_{k;t'}}^{t_0}]=F_{v}^{t_0}[D{\bf g}_{x_{k;t'}}^{t_0}]=b_kd_k.
\]
Since $d_k>0$, we obtain $a_k=b_k$.
\end{proof}

\begin{lemma}\label{lem:v-v1}
   Let $u=\prod_{i=1}^nx_{i;t}^{a_i}$ and $v=\prod_{i=1}^nx_{i;t}^{b_i}$ be two cluster monomials of $\mathcal A$ in some cluster ${\bf x}_t$. If  $\mathsf P(F_u^{t_0})=\mathsf P(F_{v}^{t_0})$ and $a_k=b_k$ for some $k\in[1,n]$, then  $\mathsf P(F_{u'}^{t_0})=\mathsf P(F_{v'}^{t_0})$, where $u'=u/x_{k;t}^{a_k}=\prod_{i\neq k}^nx_{i;t}^{a_i}$ and $v'=v/x_{k;t}^{b_k}=\prod_{i\neq k }^nx_{i;t}^{b_i}$.
\end{lemma}

\begin{proof}
We know that $u=u'\cdot x_{k;t}^{a_k}$ and $v=v'\cdot x_{k;t}^{b_k}=v'\cdot x_{k;t}^{a_k}$. Let $H:=(F_{x_{k;t}}^{t_0})^{a_k}$. Then we know
 $F_u^{t_0}=F_{u'}^{t_0}\cdot H$ and $F_{v}^{t_0}=F_{v'}^{t_0}\cdot H$. By Proposition \ref{pro:GKZ}, we have $\mathsf P(F_u^{t_0})=\mathsf P(F_{u'}^{t_0})+\mathsf P(H)$ and $\mathsf P(F_{v}^{t_0})=\mathsf P(F_{v'}^{t_0})+\mathsf P(H)$. Since $\mathsf P(F_u^{t_0})=\mathsf P(F_{v}^{t_0})$ and by Corollary \ref{cor:cancel-law}, we get $\mathsf P(F_{u'}^{t_0})=\mathsf P(F_{v'}^{t_0})$.
\end{proof}

\begin{lemma}[{\cite[Corollary 4.21]{Cao-2023}}] \label{lem:uu1}
    Let $u$ and $v$ be two cluster monomials of $\mathcal A$.
  If $\mathsf P(F_u^{t_0})=\mathsf P(F_{v}^{t_0})$, then the product $uv$ remains a cluster monomial of $\mathcal A$.  
\end{lemma}

\begin{proof}
    Since $u$ and $v$ are cluster monomials and by Theorem \ref{thm:F-compatible}, we have
\[
(u\mid\mid u)_F=2F_u^{t_0}[D{\bf g}_u^{t_0}]=0\;\;\;\text{and}\;\;\;(v\mid\mid v)_F=2F_{v}^{t_0}[D{\bf g}_{v}^{t_0}]=0.
\]
We get $F_u^{t_0}[D{\bf g}_u^{t_0}]=0=F_{v}^{t_0}[D{\bf g}_{v}^{t_0}]$.
Since $F_u^{t_0}$ and $F_{v}^{t_0}$ have the same Newton polytope, we have $F_u^{t_0}[{\bf r}]=F_{v}^{t_0}[{\bf r}]$ for any ${\bf r}\in\mathbb Z^n$. Hence,
we have
\[
(u\mid\mid v)_F=F_u^{t_0}[D{\bf g}_{v}^{t_0}]+F_{v}^{t_0}[D{\bf g}_u^{t_0}]=F_{v}^{t_0}[D{\bf g}_{v}^{t_0}]+F_u^{t_0}[D{\bf g}_u^{t_0}]=0.
\]
Then by Theorem \ref{thm:F-compatible}, we know that the product $uv$ is  a cluster monomial. 
\end{proof}

Denote by $\mathcal X$ the set of  cluster variables of $\mathcal A$. Given a cluster monomial $u={\bf x}_t^{\bf h}=\prod x_{i;t}^{h_i}$ in a seed $({\bf x}_t, B_t)$, we define its {\em support set} $\supp(u)$ by
\[\supp(u):=\{x_{i;t}\mid i\in[1,n], \;h_i\neq 0\}\subseteq \mathcal X.\] We remark that this set only depends on $u$, not on the choice of $t\in\mathbb T_n$ and ${\bf h}\in\mathbb Z^{n}$ such that $u={\bf x}_t^{\bf h}$.

\begin{theorem}
Let $\mathcal A$ be a skew-symmetrizable cluster algebra with initial seed $({\bf x}_{t_0}, B_{t_0})$. Let $u$ and $v$ be two cluster monomials in non-initial cluster variables.
 If the two $F$-polynomials $F_u^{t_0}$ and $F_{v}^{t_0}$ have the same Newton polytope, i.e., $\mathsf P(F_u^{t_0})=\mathsf P(F_{v}^{t_0})$, then $u=v$.   
\end{theorem}

\begin{proof}
Since $\mathsf P(F_u^{t_0})=\mathsf P(F_{v}^{t_0})$ and by Lemma \ref{lem:uu1}, the product $uv$ remains a cluster monomial. Thus its support set $\supp(uv)$ is a partial cluster of $\mathcal A$.

Let $[{\bf x}_{s}]$ be the left Bongartz completion of $\supp(uv)$ with respect to vertex $t_0$. 
Then $\supp(uv)\subseteq [{\bf x}_s]$. We can write $u=\prod_{i=1}^nx_{i;s}^{a_i}$ and $v=\prod_{i=1}^nx_{i;s}^{b_i}$. Thus $uv=\prod_{i=1}^nx_{i;s}^{a_i+b_i}$. We know that $a_i+b_i>0$ if and only if $x_{i;s}\in\supp(uv)$.

Since $[{\bf x}_{s}]$ is the left Bongartz completion of $\supp(uv)$ with respect to vertex $t_0$, 
we know that the $j$-th column of the $C$-matrix $C_s^{t_0}$ is non-negative for any $j$ with $x_{j;s}\in [{\bf x}_s]\setminus \supp(uv)$. For the signs of the other columns of $C_s^{t_0}$, we have the following claim.

{\em Claim:} there exist some $k$ with $x_{k;s}\in \supp(uv)$ such that the $k$-th column of $C_s^{t_0}$ is non-positive.
Otherwise, we know that $C_s^{t_0}$ is a non-negative matrix. Then by Proposition \ref{prop:C-positive}, we have $[{\bf x}_s]=[{\bf x}_{t_0}]$. Thus $\supp(uv)\subseteq [{\bf x}_s]=[{\bf x}_{t_0}]$. This contradicts that $u$ and $v$ are cluster monomials in non-initial cluster variables. This finishes the proof of the claim.

Now we prove $u=v$ by reducing the number $|\supp (uv)|$ of cluster variables in $\supp(uv)$.  By the claim, we can take an integer $k$
with $x_{k;s}\in\supp(uv)$ such that $k$-th column of $C_s^{t_0}$ is non-positive.
Since $\mathsf P(F_u^{t_0})=\mathsf P(F_{v}^{t_0})$ and by Lemma \ref{lem:ak=bk}, we know that $a_k=b_k$. 
Then by Lemma \ref{lem:v-v1}, we have
\[ \mathsf P(F_{u'}^{t_0})=\mathsf P(F_{v'}^{t_0}),\]
where $u'=u/x_{k;s}^{a_k}=\prod_{i\neq k}x_{i;s}^{a_i}$ and $v'=v/x_{k;s}^{b_k}=\prod_{i\neq k}x_{i;s}^{b_i}$.  Notice that we have 
\[ u=u'\cdot x_{k;s}^{a_k},\;\;\;v=v'\cdot x_{k;s}^{b_k}=v'\cdot x_{k;s}^{a_k}\;\;\;\text{and}\;\;\; \supp(u'v')\subsetneq \supp(uv).\]
Thus to show $u=v$, it suffices to show $u'=v'$. Now $u'$ and $v'$ are two cluster monomials in non-initial cluster variables satisfying $\mathsf P(F_{u'}^{t_0})=\mathsf P(F_{v'}^{t_0})$ and $|\supp(u'v')|< |\supp(uv)|$. Thus, by repeatedly applying the reduction argument, we complete the proof.
\end{proof}

\section{$F$-invariant and Newton polytopes in $\tau$-tilting theory}

\subsection{$F$-invariant of decorated modules} \label{sec:F-tau-tilting}

Recall that for two modules $M,N\in\mod A$, we denote by $\hom_A(M,N)=\dim_{\mathsf k}\Hom_A(M,N)$.

\begin{definition}[$E$-invariant and partial $E$-invariant, \cite{DWZ10,DK-2015,air_2014}]
    Let $\mathcal M=(M,P)$ and $\mathcal N=(N,Q)$ be two decorated $A$-modules. The {\em partial $E$-invariant} $E^{\rm proj}(\mathcal M,\mathcal N)$ and the {\em $E$-invariant} $E^{\rm sym}(\mathcal M,\mathcal N)$ of the ordered pair $(\mathcal M,\mathcal N)$ are defined as follows:
\begin{eqnarray}
   E^{\rm proj}(\mathcal M,\mathcal N)&:=&\hom_A(N,\tau M)+\hom_A(P,N),\\
  E^{\rm sym}(\mathcal M,\mathcal N)&:=& E^{\rm proj}(\mathcal M,\mathcal N)+ E^{\rm proj}(\mathcal N,\mathcal M)\\
  &=&\hom_A(N,\tau M)+\hom_A(P,N)+\hom_A(M,\tau N)+\hom_A(Q,M).\nonumber
\end{eqnarray}

\end{definition}

\begin{definition}[$F$-invariant and partial $F$-invariant]
   Let $\mathcal M=(M,P)$ and $\mathcal N=(N,Q)$ be two decorated $A$-modules.  
   Let $F_{\mathcal M}=F_M=\sum_{{\bf v}\in\mathbb N^n}c_{\bf v}{\bf y}^{\bf v}\in\mathbb Z[y_1,\ldots,y_n]$ be the $F$-polynomial of $\mathcal M$ and
 ${\bf g}_{\mathcal N}\in\mathbb Z^n$ the $g$-vector of $\mathcal N$. The {\em partial $F$-invariant} $F_{\mathcal M}[{\bf g}_{\mathcal N}]$ and the {\em $F$-invariant} $(\mathcal M\mid\mid \mathcal N)_F$ of the ordered pair $(\mathcal M,\mathcal N)$ are defined as follows:
    \begin{eqnarray}
   F_{\mathcal M}[{\bf g}_{\mathcal N}]&:=&\max\{\langle {\bf v},{\bf g}_{\mathcal N}\rangle\mid c_{\bf v}\neq 0\}=\max\{\langle {\bf v},{\bf g}_{\mathcal N}\rangle\mid {\bf v}\in \mathsf P(M)\},\\
   (\mathcal M\mid\mid \mathcal N)_F&:=& F_{\mathcal M}[{\bf g}_{\mathcal N}]+ F_{\mathcal N}[{\bf g}_{\mathcal M}].
\end{eqnarray}
\end{definition}

For a module $X\in\mod A$, we denote by $\underline{\dim} X\in\mathbb N^n$ the dimension vector of $X$.

\begin{lemma}[{\cite[Prop. 2.4]{air_2014}}, {\cite[Thm. 1.4]{AR-1985}}] \label{lem:g-MX}
 Let $M$ and $X$ be two modules in $\mod A$. Then\footnote{Note that $g$-vectors of modules in this paper are the negative of the $g$-vectors used in \cite{air_2014}.} 
 \[ \langle {\bf g}_M, \underline{\dim} X\rangle=\hom_A(X,\tau M)-\hom_A(M, X).
 \]
\end{lemma}

\begin{lemma}\label{lem:N0-Nf}
Let $M\in\mod A$ be a $\tau$-rigid module and let \begin{eqnarray}\label{eqn:tf-seq}
    \xymatrix{0\ar[r]&N_t\ar[r]& N\ar[r]&N_f\ar[r]&0,}
\end{eqnarray}
be the canonical sequence of a module $N$ with respect to the torsion pair  $(\Fac M, M^\bot)$. Then  for any quotient module $N_0$ of $N$, we have 
\[
\langle {\bf g}_M, \underline{\dim} N_0\rangle\leq \langle{\bf g}_M, \underline{\dim} N_f\rangle=
 \hom(N,\tau M).
\]
\end{lemma}

\begin{proof}
 We know that $N_t\in\Fac M$ and $N_f\in M^{\bot}$.  Since $N_f\in M^{\bot}$ and by Lemma \ref{lem:g-MX}, we have
\[\langle {\bf g}_M, \underline{\dim} N_f\rangle=\hom_A(N_f,\tau M)-\hom_A(M,N_f)=\hom_A(N_f,\tau M).\]
Since $N_t\in \Fac M\subseteq \prescript{\bot}{}(\tau M)$ and by applying the functor $\Hom_A(-,\tau M)$ to \eqref{eqn:tf-seq}, we get $\Hom_A(N_f,\tau M)\cong \Hom_A(N,\tau M)$. Hence, we have 
\[\langle {\bf g}_M,\underline{\dim} N_f\rangle=\hom_A(N_f,\tau M)=\hom_A(N,\tau M).\]

Suppose that $N_0$ is a quotient module of $N$, i.e., we have an exact sequence $N\to N_0\to 0$. By applying the functor $\Hom_A(-,\tau M)$, we see that 
\[ \hom(N_0,\tau M)\leq \hom(N,\tau M).\]
By Lemma \ref{lem:g-MX}, we have
\begin{eqnarray*}\label{eqn:N0-Nf}
    \langle {\bf g}_M, \underline{\dim} N_0\rangle&=&\hom_A(N_0,\tau M)-\hom_A(M,N_0)\\
    &\leq& \hom_A(N_0,\tau M) \\
    &\leq&  \hom_A(N,\tau M)=\langle{\bf g}_M, \underline{\dim} N_f\rangle.
\end{eqnarray*}
This completes the proof.
\end{proof}

\begin{lemma}\label{lem:N0-Nt}
Let $M\in\mod A$ be a $\tau$-rigid module and  let \begin{eqnarray}\label{eqn:tf-seq-1}
\xymatrix{0\ar[r]&N_t\ar[r]& N\ar[r]&N_f\ar[r]&0,}
\end{eqnarray}
be the canonical sequence of a module $N$ with respect to the torsion pair  $(\prescript{\bot}{}{(\tau M)}, \Sub \tau M)$. Then for any submodule $N_0$ of $N$, we have 
\[
\langle -{\bf g}_M, \underline{\dim} N_0\rangle\leq \langle -{\bf g}_M, \underline{\dim} N_t\rangle=
 \hom(M, N).
\]
\end{lemma}
\begin{proof}
     We know that $N_t\in\prescript{\bot}{}{(\tau M)}$ and $N_f\in \Sub \tau M$.  Since $N_t\in \prescript{\bot}{}{(\tau M)}$ and by Lemma \ref{lem:g-MX}, we have
\[\langle {-\bf g}_M, \underline{\dim} N_t\rangle=-\hom_A(N_t,\tau M)+\hom_A(M,N_t)=\hom_A(M,N_t).\]
Since $\Hom_A(M,\tau M)=0$, $N_f\in \Sub \tau M$ and by applying the functor $\Hom_A(M,-)$ to \eqref{eqn:tf-seq-1}, we get $\Hom_A(M,N_t)\cong \Hom_A(M,N)$. Hence, we have 
\[\langle -{\bf g}_M,\underline{\dim} N_t\rangle=\hom_A(M,N_t)=\hom_A(M,N).\]

Suppose that $N_0$ is a submodule of $N$, i.e., we have an exact sequence $0\to N_0\to N$. By applying the functor $\Hom_A(M,-)$, we see that 
\[ \hom(M,N_0)\leq \hom(M,N).\]
By Lemma \ref{lem:g-MX}, we have
\begin{eqnarray*}\label{eqn:N0-Nf}
    \langle -{\bf g}_M, \underline{\dim} N_0\rangle&=&-\hom_A(N_0,\tau M)+\hom_A(M,N_0)\\
    &\leq& \hom_A(M, N_0) \\
    &\leq&  \hom_A(M, N)=\langle- {\bf g}_M, \underline{\dim} N_t\rangle.
\end{eqnarray*}
This completes the proof.
\end{proof}

\begin{proposition}\label{pro:check-F}
    Let $M\in\mod A$ be a $\tau$-rigid module. Then for any module $N\in\mod A$, we have 
    \[ F_N[{\bf g}_M]=\hom_A(N,\tau M)\quad\text{and}\quad \check F_N[-{\bf g}_M]=\hom_A(M,N),
    \]
    where $\check{F}_N$ is the dual $F$-polynomial of $N$. 
\end{proposition}
\begin{proof}
    This follows from Lemma \ref{lem:N0-Nf} and Lemma \ref{lem:N0-Nt}.
\end{proof}

\begin{theorem}
\label{thm:FE-hom}
Let $\mathcal M=(M,P)$ be a $\tau$-rigid pair in $\mod A$. Then for any decorated $A$-module $\mathcal N=(N,Q)$, we have 
 \[F_{\mathcal N}[{\bf g}_{\mathcal M}]=E^{\rm proj}(\mathcal M,\mathcal N)=\hom_A(N,\tau M)+\hom_A(P,N).\] 
\end{theorem}
\begin{proof}
Since $F_{\mathcal N}=F_N$, it suffices to show that $F_{N}[{\bf g}_{\mathcal M}]=\hom_A(N,\tau M)+\hom_A(P,N)$. For any quotient module $N_0$ of $N$, we have $\hom_A(P,N_0)\leq \hom_A(P,N)$. By Lemma \ref{lem:g-MX}, we have $\langle {\bf g}_{P},\underline{\dim} N_0\rangle=\hom_A(N_0,\tau P)-\hom_A(P, N_0)=-\hom_A(P, N_0)$.
Thus
\begin{eqnarray*}
    \langle {\bf g}_{\mathcal M},\underline{\dim} N_0\rangle&=& \langle {\bf g}_{M},\underline{\dim} N_0\rangle-\langle {\bf g}_{P},\underline{\dim} N_0\rangle \quad\;\;(\text{by\;} {\bf g}_{\mathcal M}={\bf g}_M-{\bf g}_{P})\\
    &=&\langle {\bf g}_{M},\underline{\dim} N_0\rangle+ \hom_A(P,N_0)\\
    &\leq&\langle {\bf g}_{M},\underline{\dim} N_0\rangle+\hom_A(P,N)\\
    &\leq& \hom_A(N,\tau M)+\hom_A(P,N) \quad\;\text{(by Lemma \ref{lem:N0-Nf}).}
\end{eqnarray*}
Since $N_0$ is an arbitrary quotient module of $N$, we have \[F_N[{\bf g}_{\mathcal M}]\leq \hom_A(N,\tau M)+\hom_A(P,N).\]

Now let us show the converse inequality. Consider the canonical sequence of $N$ with respect to the torsion pair  $(\Fac M, M^\bot)$:
\begin{eqnarray}\label{eqn:tf-seq-b}
\xymatrix{0\ar[r]&N_t\ar[r]& N\ar[r]&N_f\ar[r]&0,}
\end{eqnarray}
where $N_t\in\Fac M$ and $N_f\in M^\bot$. 

Since $\mathcal M=(M,P)$ is a $\tau$-rigid pair, we have $\Hom_A(P,M)=0$ and thus $\Hom_A(P, N_t)=0$. Applying the functor $\Hom_A(P,-)$ to \eqref{eqn:tf-seq-b}, we see $\Hom_A(P,N)\cong \Hom_A(P, N_f)$. By Lemma \ref{lem:g-MX}, we have
\begin{eqnarray*}
    \langle {\bf g}_P, \underline{\dim} N_f\rangle&=&\hom_A(N_f,\tau P)-\hom_A(P, N_f)\\
    &=&-\hom_A(P, N_f)=-\hom_A(P, N).
\end{eqnarray*}
By Lemma \ref{lem:N0-Nf}, we know that $ \langle {\bf g}_{M},\underline{\dim} N_f\rangle=\hom_A(N,\tau M)$. Thus
\[
 \langle {\bf g}_{\mathcal M},\underline{\dim} N_f\rangle= \langle {\bf g}_{M},\underline{\dim} N_f\rangle- \langle {\bf g}_{P},\underline{\dim} N_f\rangle=\hom_A(N,\tau M)+\hom_A(P, N).
\]
Since $N_f$ is a quotient module of $N$, we have $\langle {\bf g}_{\mathcal M},\underline{\dim} N_f\rangle\leq F_N[{\bf g}_{\mathcal M}]$. Thus
\[
\hom_A(N,\tau M)+\hom_A(P, N)=\langle {\bf g}_{\mathcal M},\underline{\dim} N_f\rangle\leq F_N[{\bf g}_{\mathcal M}].
\]
Hence, we have $F_{\mathcal N}[{\bf g}_{\mathcal M}]=F_N[{\bf g}_{\mathcal M}]=\hom_A(N,\tau M)+\hom_A(P,N)=E^{\rm proj}(\mathcal M,\mathcal N)$.
\end{proof}

\begin{remark}
  We remark that the above theorem can be deduced from \cite[Theorem 3.6]{fei_2019b}. However, the proofs here and that in \cite{fei_2019b} are quite different.
\end{remark}

\begin{corollary}\label{cor:additive}
Let $\mathcal M=\mathcal M'\oplus \mathcal M''$ be a $\tau$-rigid pair. Then for any decorated $A$-module $\mathcal N=\mathcal N'\oplus \mathcal N''$, we have
\[ F_{\mathcal N}[{\bf g}_{\mathcal M}]=F_{\mathcal N}[{\bf g}_{\mathcal M'}]+F_{\mathcal N}[{\bf g}_{\mathcal M''}]\;\;\;\text{and}\;\;\;F_{\mathcal N}[{\bf g}_{\mathcal M}]=F_{\mathcal N'}[{\bf g}_{\mathcal M}]+F_{\mathcal N''}[{\bf g}_{\mathcal M}].
\]
\end{corollary}
\begin{proof}
It is easy to check that $E^{\rm proj}(\mathcal M, \mathcal N)=E^{\rm proj}(\mathcal M', \mathcal N)+E^{\rm proj}(\mathcal M'', \mathcal N)$ and $E^{\rm proj}(\mathcal M, \mathcal N)=E^{\rm proj}(\mathcal M, \mathcal N')+E^{\rm proj}(\mathcal M, \mathcal N'')$. Then the desired results follow from Theorem \ref{thm:FE-hom}.  
\end{proof}

\begin{corollary}\label{cor-F=E}
Let $\mathcal M=(M,P)$ and $\mathcal N=(N,Q)$ be two $\tau$-rigid pairs in  $\mod A$. Then the following statements hold.
\begin{itemize}
    \item [(i)] We have $(\mathcal M\mid\mid \mathcal N)_F=F_{\mathcal M}[{\bf g}_{\mathcal N}]+F_{\mathcal N}[{\bf g}_{\mathcal M}]=E^{\rm sym}(\mathcal M,\mathcal N)$.
    \item[(ii)] The direct sum $\mathcal M\oplus \mathcal N$ is a $\tau$-rigid pair in $\mod A$ if and only if $(\mathcal M\mid\mid \mathcal N)_F=0$.
\end{itemize}
\end{corollary}
\begin{proof}
 (i) By Theorem \ref{thm:FE-hom}, we have $F_{\mathcal N}[{\bf g}_{\mathcal M}]=E^{\rm proj}(\mathcal M,\mathcal N)$ and $F_{\mathcal M}[{\bf g}_{\mathcal N}]=E^{\rm proj}(\mathcal N, \mathcal M)$. Thus 
 \[(\mathcal M\mid\mid \mathcal N)_F=F_{\mathcal N}[{\bf g}_{\mathcal M}]+F_{\mathcal M}[{\bf g}_{\mathcal N}]=E^{\rm proj}(\mathcal M,\mathcal N)+E^{\rm proj}(\mathcal N,\mathcal M) =E^{\rm sym}(\mathcal M,\mathcal N).\]

 (ii) Since $E^{\rm sym}(\mathcal M,\mathcal N)=\hom_A(M,\tau N)+\hom_A(Q,M)+\hom_A(N,\tau M)+\hom_A(P,N)$, we know that $\mathcal M\oplus \mathcal N$ is $\tau$-rigid if and only if $E^{\rm sym}(\mathcal M,\mathcal N)=0$. Then the desirted result follows from (i).
\end{proof}

\begin{corollary}\label{cor:dom-hom}
Let $\mathcal M=(M,P)$ be a $\tau$-rigid pair in $\mod A$. Then a module $N\in \mod A$ belongs to   $\prescript{\bot}{}(\tau M)\cap P^{\bot}$ if and only if $F_N[{\bf g}_{\mathcal M}]=0$, that is, $\langle{\bf g}_{\mathcal M},\underline{\dim} N_0\rangle\leq 0$ for any quotient module $N_0$ of $N$.
\end{corollary}
\begin{proof}
By Theorem \ref{thm:FE-hom}, we have $F_N[{\bf g}_{\mathcal M}]=\hom_A(N,\tau M)+\hom_A(P,N)$. Then the result follows.
\end{proof}

\subsection{Newton polytopes of $\tau$-rigid modules}

Recall that we use $\mathsf P(M)$ to denote the Newton polytope of a module $M\in\mod A$.
\begin{lemma}[{\cite[Theorem 3.2]{cao-2025b}}] \label{lem:sum-rigid}
 Let $U$ and $V$ be two $\tau$-rigid modules in $\mod A$. If  $\mathsf P(U)=\mathsf P(V)$, the direct sum $U\oplus V$ remains $\tau$-rigid in $\mod A$.    
\end{lemma}
\begin{proof}
  This result is first proved in \cite{cao-2025b}. Here we give a different proof.  Since $U$ and $V$ are $\tau$-rigid modules, we know that $E^{\rm proj}(U,U)=\hom_A(U,\tau U)=0$ and $E^{\rm proj}(V,V)=\hom_A(V,\tau V)=0$.
Then by Theorem \ref{thm:FE-hom},
we obtain $F_U[{\bf g}_U]= E^{\rm proj}(U,U)=0$ and $F_V[{\bf g}_V]=E^{\rm proj}(V,V)=0$.

  Since $U$ and $V$ have the same Newton polytope, we know that $F_U[{\bf r}]=F_V[{\bf r}]$ for any ${\bf r}\in\mathbb R^n$. In particular, we have $F_U[{\bf g}_V]=F_V[{\bf g}_V]=0$ and $F_V[{\bf g}_U]=F_U[{\bf g}_U]=0$. Thus 
  \[(U\mid\mid V)_F=F_U[{\bf g}_V]+F_V[{\bf g}_U]=0.\]
By Corollary \ref{cor-F=E} (ii), we know that $U\oplus V$ is a $\tau$-rigid module.
\end{proof}

\begin{lemma}\label{lem:exchange}
 Let $\mathcal M=\oplus_{i=1}^n \mathcal M_i$ be a basic $\tau$-tilting pair in $\mod A$ and $\mathcal M'=\mu_k(\mathcal M)=(\oplus_{i\neq k}\mathcal M_i)\oplus \mathcal M_k'$ a left mutation of $\mathcal M$. Then we have  
 \begin{itemize}
     \item [(i)]  $E^{\rm proj}(\mathcal M_k, \mathcal M_k')=0$ and $E^{\rm proj}(\mathcal M_k',\mathcal M_k)>0$. 
     \item [(ii)] $F_{\mathcal M_k'}[{\bf g}_{\mathcal M_k}]=0$ and $F_{\mathcal M_k}[{\bf g}_{\mathcal M_k'}]>0$.
 \end{itemize}
\end{lemma}

\begin{proof}
(i) We write $\oplus_{i\neq k}\mathcal M_i=(U,Q)$, $\mathcal M_k=(M_k, Q_k)$ and $\mathcal M_k'=(M_k', Q_k')$. Thus \[\mathcal M=(U\oplus M_k, Q\oplus Q_k)\;\;\;\text{and}\;\;\;\mathcal M'=(U\oplus M_k', Q\oplus Q_k').\]
Since $\mathcal M_k$ is  indecomposable, we know that either $M_k$ or $Q_k$ is zero. Similarly, either $M_k'$ or $Q_k'$ is zero. 

Since the mutation $\mathcal M'=\mu_k(\mathcal M)$ is a left mutation, we have $\Fac (U\oplus M_k')\subsetneq \Fac (U\oplus M_k)$.
This implies that $M_k\neq 0$. Thus $\mathcal M_k=(M_k, 0)$ and  $\mathcal M=(U\oplus M_k, Q)$. Since $\mathcal M'=(U\oplus M_k', Q\oplus Q_k')$ and $\mathcal M=(U\oplus M_k, Q)$ are $\tau$-tilting pairs and by Theorem \ref{thm-air-mutation} (i), we know that 
\begin{eqnarray*}
    \Fac (U\oplus M_k')&=&\prescript{\bot}{}{(\tau U\oplus \tau M_k')}\cap (Q\oplus Q_k')^\bot=[\prescript{\bot}{}{(\tau U)}\cap Q^\bot] \cap [\prescript{\bot}{}{(\tau M_k')}\cap (Q_k')^\bot],\\
    \Fac (U\oplus M_k)&=&\prescript{\bot}{}{(\tau U\oplus \tau M_k)}\cap Q^\bot=[\prescript{\bot}{}{(\tau U)}\cap Q^\bot] \cap \prescript{\bot}{}{(\tau M_k)},
\end{eqnarray*}
Since  $M_k'\in\Fac (U\oplus M_k')\subsetneq \Fac(U\oplus M_k)$ and $M_k\in \Fac(U\oplus M_k)\setminus\Fac(U\oplus M_k')$,
we have
\[ M_k'\in \prescript{\bot}{}{(\tau M_k)}\;\;\;\text{and}\;\;\; M_k\notin \prescript{\bot}{}{(\tau M_k')}\cap (Q_k')^\bot.
\]
Thus we have $E^{\rm proj}(\mathcal M_k, \mathcal M_k')=\hom_A(M_k',\tau M_k)=0$ and
\[
E^{\rm proj}(\mathcal M_k', \mathcal M_k)=\hom_A(M_k,\tau M_k')+\hom_A(Q_k', M_k)>0.
\]

(ii) This follows from (i) and Theorem \ref{thm:FE-hom}.
\end{proof}

\begin{theorem}
 Let $U$ and $V$ be two $\tau$-rigid modules in $\mod A$. If $U$ and $V$ have the same Newton polytope, i.e., $\mathsf P(U)=\mathsf P(V)$, then $U\cong V$.    
\end{theorem}
\begin{proof}
By Lemma \ref{lem:sum-rigid}, we know that the direct sum $U\oplus V$ is $\tau$-rigid.  We will prove that $U\cong V$ by reducing the number $|U\oplus V|$ of iso-classes of indecomposable direct summands of $U\oplus V$.
We may assume $|U\oplus V|>0$; otherwise, we have $U=V=0$ and the desired result holds trivially.

Let $(U\oplus V)^\flat$ be the basic $\tau$-rigid module such that $\add (U\oplus V)^\flat=\add U\oplus V$. Let $\mathcal M=(M,P)$ be the left Bongartz completion of $(U\oplus V)^\flat$, which is the basic $\tau$-tilting pair satisfying that $(U\oplus V)^\flat$ is a direct summand of $M$ and $\Fac M=\Fac (U\oplus V)$. 

Let us write $\mathcal M=\oplus_{i=1}^n\mathcal M_i$, where each $\mathcal M_i$ is an indecomposable $\tau$-rigid pair. We can assume that 
\begin{eqnarray}\label{eqn:MN-ab}
  \mathcal U:=(U,0)=\oplus_{i=1}^n \mathcal M_i^{a_i}\;\;\; \text{and}\;\;\;\mathcal V:=(V,0)=\oplus_{i=1}^n \mathcal M_i^{b_i},  
\end{eqnarray}
where $a_i, b_i\in\mathbb Z_{\geq 0}$.

Since $|U\oplus V|>0$, we know that $\{0\}\subsetneq \Fac M=\Fac (U\oplus V)$. Then by Proposition \ref{pro-air}, there exists a left mutation $\mathcal M'=(M',P')=\mu_k(\mathcal M)$ of $\mathcal M$ such that $\{0\}\subseteq \Fac M'\subsetneq \Fac M$. 

We claim that $\mathcal M_k$ is a direct summand of $((U\oplus V)^\flat,0)$. Otherwise,  $((U\oplus V)^\flat,0)$ is also a direct summand of $\mathcal M'=(M', P')$. Then by Remark \ref{rmk:left-completion} and the fact that  $\mathcal M=(M,P)$ is the left Bongartz completion of $(U\oplus V)^\flat$, we have $\Fac M\subseteq \Fac M'$. This contradicts  $\Fac M'\subsetneq \Fac M$.

Since $\mathcal M_k$ is a direct summand of $((U\oplus V)^\flat,0)$, we can assume that $\mathcal M_k=(M_k,0)$, where $M_k$ is an indecomposable direct summand of $(U\oplus V)^\flat$. Let $\mathcal M_k'$ be the new indecomposable $\tau$-rigid pair obtained in $\mathcal M'=\mu_k(\mathcal M)$.

By Lemma \ref{lem:exchange}, we know that $F_{\mathcal M_k}[{\bf g}_{\mathcal M_k'}]>0$.
For $i\neq k$, we know that $\mathcal M_i\oplus \mathcal M_k'$ is $\tau$-rigid. Then by Corollary \ref{cor-F=E} (ii), we have $F_{\mathcal M_i}[{\bf g}_{\mathcal M_k'}]=0$ for any $i\neq k$.

By \eqref{eqn:MN-ab} and Proposition \ref{pro:DWZ-FmFn}, we have $F_{\mathcal U}=\prod_{i=1}^nF_{\mathcal M_i}^{a_i}$ and $F_{\mathcal V}=\prod_{i=1}^nF_{\mathcal M_i}^{b_i}$. Then by Corollary \ref{cor-F-prod}, we have 
\begin{eqnarray*}
 F_{\mathcal U}[{\bf g}_{\mathcal M_k'}]&=&\sum_{i=1}^na_iF_{\mathcal M_i}[{\bf g}_{\mathcal M_k'}] =a_k F_{\mathcal M_k}[{\bf g}_{\mathcal M_k'}],\\ 
 F_{\mathcal V}[{\bf g}_{\mathcal M_k'}]&=&\sum_{i=1}^nb_iF_{\mathcal M_i}[{\bf g}_{\mathcal M_k'}] =b_k F_{\mathcal M_k}[{\bf g}_{\mathcal M_k'}].
\end{eqnarray*}
Since $U$ and $V$ have the same Newton polytope, we know that $F_U[{\bf r}]=F_V[{\bf r}]$ for any ${\bf r}\in\mathbb Z^n$. In particular, we have
\[
a_k F_{\mathcal M_k}[{\bf g}_{\mathcal M_k'}]= F_{\mathcal U}[{\bf g}_{\mathcal M_k'}]= F_{U}[{\bf g}_{\mathcal M_k'}]= F_{V}[{\bf g}_{\mathcal M_k'}]= F_{\mathcal V}[{\bf g}_{\mathcal M_k'}]=b_k F_{\mathcal M_k}[{\bf g}_{\mathcal M_k'}].
\]
Since $F_{\mathcal M_k}[{\bf g}_{\mathcal M_k'}]>0$, we get $a_k=b_k$. So we have 
 following decompositions \[U=U'\oplus M_k^{a_k},\quad V=V'\oplus M_k^{b_k}=V'\oplus M_k^{a_k}.\] Note that $U'$ and $V'$ do not contain $M_k$ as  a direct summand. In particular, we have \[|U'\oplus V'|<|U\oplus V|.\]
By Corollary \ref{cor:minkowski-sum}, we know that $\mathsf P(U)=\mathsf P(U')+\mathsf P(M_k^{a_k})$ and $\mathsf P(V)=\mathsf P(V')+\mathsf P(M_k^{a_k})$. Since $\mathsf P(U)=\mathsf P(V)$ and by Corollary \ref{cor:cancel-law}, we get $\mathsf P(U')=\mathsf P(V')$.

In order to show $U\cong V$, it suffices to show $U'\cong V'$.
Now $U'$ and $V'$ are two $\tau$-rigid modules satisfying  $\mathsf P(U')=\mathsf P(V')$ and $|U'\oplus V'|<|U\oplus V|$. We can therefore apply the reduction and complete the proof.
\end{proof}

\subsection{Newton polytopes of left finite multi-semibricks}
A module $M\in\mod A$ is called {\em left finite}, if the smallest torsion class $\langle M\rangle_{\rm tors }$ containing $M$ is a functorially finite torsion class in $\mod A$. A module $C\in\mod A$ is called a {\em brick}, if $\End_A(C)\cong \mathsf k$.  A set $\Omega$ of iso-classes of bricks is called a {\em semibrick}, if $\Hom_A(C,C')=0$ for any $C,C'\in\Omega$ with $C\neq C'$. A semibrick $\Omega$ is {\em left finite}, if the smallest torsion class $\langle \Omega\rangle_{\rm tors}$ containing the bricks in $\Omega$ is a functorially finite torsion class.

\begin{proposition}[{\cite[Proposition 2.9]{Asai-semibrick}}] \label{pro:asai-bijection}
The map $\Omega\mapsto \langle \Omega\rangle_{\rm tors}$ gives a bijection from the left finite semibricks in $\mod A$ to functorially finite torsion classes in $\mod A$. In particular, the left finite semibricks in $\mod A$ are in bijection with the basic $\tau$-tilting pairs in $\mod A$.
\end{proposition}

Let $\mathcal M=(M,Q)$ be a basic $\tau$-tilting pair in $\mod A$ and $\Omega$ a left finite semibrick in $\mod A$. The set $\Omega$ is called the {\em labeling semibrick} of $\mathcal M=(M,Q)$, if $\Fac M=\langle\Omega\rangle_{\rm tors}$ holds. This name is justified by the following result.

\begin{proposition}[{\cite[Lemma 2.5, Proposition 2.13]{Asai-semibrick}}] \label{pro:asai-brick}
    Let $\mathcal M=\oplus_{j=1}^n(M_j, Q_j)$ a basic $\tau$-tilting pair in $\mod A$. Denote by $I\subseteq [1,n]$ the subset such that $i\in I$ if and only if $\mu_i(\mathcal M)$ is a left mutation of $\mathcal M$.
    Then the following statements hold.
    \begin{itemize}
        \item [(i)] For each $k\in I$, there exists a unique brick $C_k$, called a labeling brick, satisfying that 
\begin{eqnarray*}
            C_k\in\Fac M_k\quad\text{and}\quad \Hom_A(M_j,C_k)=0\quad \forall j\in[1,n]\setminus k,
        \end{eqnarray*}
        i.e., $C_k$ belongs to $(\oplus_{j\neq k}M_j)^\bot\cap \Fac M_k$.
        \item [(ii)] Let $\Omega:=\{C_i\}_{i\in I}$ be the set of labeling bricks for the left mutations of $\mathcal M$. Then $\Omega$ is a left finite semibrick and we have $\langle \Omega\rangle_{\rm tors}=\Fac M$.
    \end{itemize}
\end{proposition}

\begin{lemma}\label{lem:multi-brick}
Let $\mathcal M=\oplus_{j=1}^n \mathcal M_j$ be a basic $\tau$-tilting pair in $\mod A$ and let $\Omega=\{C_i\}_{i\in I}$ be the labeling semibrick of $\mathcal M$. Fix a labeling brick $C_k\in\Omega$ and let  $\mathcal M'=\mu_k(\mathcal M)=(\oplus_{j\neq k}\mathcal M_j)\oplus\mathcal M_k'$ be the corresponding left mutation of $\mathcal M$. Then the following statements hold.
\begin{itemize}
    \item [(i)] Denote by $\mathcal M_k'=(M_k',Q_k')$. Then  $C_k\notin \prescript{\bot}{}(\tau M_k')\cap Q_k'^{\bot}$ and $C_i\in \prescript{\bot}{}(\tau M_k')\cap Q_k'^{\bot}$ for any $i\in I\setminus \{k\}$.
    \item[(ii)] We have $F_{C_k}[{\bf g}_{{\mathcal M}_k'}]\neq 0$ and  $F_{C_i}[{\bf g}_{{\mathcal M}_k'}]=0$ for any $i\in I\setminus\{k\}$.
\end{itemize}
\end{lemma}
\begin{proof}
 (i) Let us write 
$\mathcal M_j=(M_j,Q_j)$ for $j\in[1,n]$ and $(U,Q):=\oplus_{j\neq k}(M_j,Q_j)$.
  Then \[\mathcal M=(U\oplus M_k,Q\oplus Q_k)\quad \text{and}\quad \mathcal M'=(U\oplus M_k',Q\oplus Q_k').\]
 Since $\mathcal M'=\mu_k(\mathcal M)$ is a left mutation, we have
\begin{eqnarray*}
   \Fac (U\oplus M_k')=\Fac U\quad\text{and}\quad  \Fac (U\oplus M_k)=\prescript{\bot}{}(\tau U)\cap Q^{\bot}. 
\end{eqnarray*}
 On the other hand, since  $\mathcal M'$ is a $\tau$-tilting pair and by Theorem \ref{thm-air-mutation} (i), we know that 
\begin{eqnarray}\label{eqn:u-nk-1}
    \Fac (U\oplus M_k')=[\prescript{\bot}{}(\tau U)\cap Q^{\bot}]\cap [\prescript{\bot}{}(\tau M_k')\cap Q_k'^{\bot}].
\end{eqnarray}

By Proposition \ref{pro:asai-brick} (i), we see that $C_k\in U^{\bot}\cap \Fac M_k$. On the one hand, we know that 
 \[C_k\in\Fac M_k\subseteq \Fac(U\oplus M_k)=\prescript{\bot}{}(\tau U)\cap Q^{\bot}.\] On the other hand, we know that $C_k\in U^\bot$ and thus $C_k\notin \Fac U=\Fac(U\oplus M_k')$. Then by \eqref{eqn:u-nk-1}, we obtain that $C_k\notin \prescript{\bot}{}(\tau M_k')\cap Q_k'^{\bot}$.

 By Proposition \ref{pro:asai-brick} (i), we know that $C_i\in \Fac M_i\subseteq \Fac U=\Fac (U\oplus M_k')$ for $i\in I\setminus \{k\}$.
 Then by \eqref{eqn:u-nk-1}, we get $C_i\in \prescript{\bot}{}(\tau M_k')\cap Q_k'^{\bot}$.
 
 (ii) This follows from (i) and  Corollary \ref{cor:dom-hom}.
\end{proof}

\begin{definition}[Multi-semibrick]
A module $M\in \mod A$ is called a {\em multi-semibrick} if
$M$ has a decomposition $M\cong \oplus_{i=1}^rC_i^{a_i}$ such that each $C_i\in\mod A$ is a brick and $\Hom_A(C_i,C_j)=0$ for any $i\neq j$. 
\end{definition}

For a module $M\in\mod A$, we define its {\em support set}  $\supp(M)$ as the set of iso-classes of indecomposable modules which appears as a direct summand of $M$. Clearly, if $M$ is a multi-semibrick, then its support set $\supp(M)$ is a semibrick.

\begin{lemma}\label{lem:UV-torsion}
    Let $U$ and $V$ be two left finite modules in $\mod A$. If $U$ and $V$ have the same Newton polytope, then $\langle U\rangle_{\rm tors}=\langle V\rangle_{\rm tors}$.
\end{lemma}
\begin{proof}
 Since $U$ and $V$ are left finite modules, there are two basic $\tau$-tilting pairs $\mathcal M=(M,P)$ and $\mathcal N=(N, Q)$ such that $\Fac M=\langle U\rangle_{\rm tors}$ and $\Fac N=\langle V\rangle_{\rm tors}$.
 Since $U\in\Fac M=\prescript{\bot}{}(\tau M)\cap P^{\bot}$, $V\in\Fac N=\prescript{\bot}{}(\tau N)\cap Q^{\bot}$ and
 by Corollary \ref{cor:dom-hom}, we have $F_U[{\bf g}_{\mathcal M}]=0$ and $F_V[{\bf g}_{\mathcal N}]=0$.

 Since $U$ and $V$ have the same Newton polytope, we know that $F_U[{\bf r}]=F_V[{\bf r}]$ for any ${\bf r}\in\mathbb Z^n$. In particular, we have $F_U[{\bf g}_{\mathcal N}]=F_V[{\bf g}_{\mathcal N}]=0$ and $F_V[{\bf g}_{\mathcal M}]=F_U[{\bf g}_{\mathcal M}]=0$.

 Then by Corollary \ref{cor:dom-hom} again, we see that $U\in \prescript{\bot}{}(\tau N)\cap Q^{\bot}=\Fac N=\langle V\rangle_{\rm tors}$ and $V\in \prescript{\bot}{}(\tau M)\cap P^{\bot}=\Fac M=\langle U\rangle_{\rm tors}$. Hence, we have $\langle U\rangle_{\rm tors}=\langle V\rangle_{\rm tors}$.
\end{proof}

\begin{lemma}\label{lem:UV-semibrick}
    Let $U$ and $V$ be two left finite multi-semibricks in $\mod A$. If $U$ and $V$ have the same Newton polytope, then $\supp(U)=\supp(V)$.
\end{lemma}

\begin{proof}
Since $U$ and $V$ are two left finite multi-semibricks, we know that $\supp(U)$ and $\supp(V)$ are left finite semibricks in $\mod A$. Since $U$ and $V$ have the same Newton polytope and by Lemma \ref{lem:UV-torsion}, we know that  $\langle U\rangle_{\rm tors}=\langle V\rangle_{\rm tors}$. Since the two left finite semibricks $\supp(U)$ and $\supp(V)$ generate the same torsion class $\langle U\rangle_{\rm tors}=\langle V\rangle_{\rm tors}$ and by Proposition \ref{pro:asai-bijection}, we have $\supp(U)=\supp(V)$.
\end{proof}

\begin{theorem}
   Let $U$ and $V$ be two left finite multi-semibricks in $\mod A$. If $U$ and $V$ have the same Newton polytope, then $U\cong V$.
\end{theorem}
\begin{proof}
  By Lemma \ref{lem:UV-semibrick}, we know that $\supp (U)=\supp(V)$. We can assume that $\supp (U)$ is non-empty. Otherwise, $U=V=0$ and the desired result holds trivially.

Since $U$ and $V$ are left finite and $\supp (U)=\supp (V)$, we know that there exists a basic $\tau$-tilting pair $\mathcal M=\oplus_{j=1}^n\mathcal M_j=(M,Q)$ such that $\Fac M=\langle U\rangle_{\rm tors}=\langle V\rangle_{\rm tors}$. Since $\supp(U)$ is non-empty, we know that $\{0\}\subsetneq \Fac M$.

Let $I$ be the subset of $[1,n]$ such that $i\in I$ if only if $\mu_i(\mathcal M)$ is a left mutation of $\mathcal M$. By Proposition \ref{pro:asai-brick}, we know that the bricks in $\supp(U)=\supp (V)$ are indexed by $I$. Say  $\supp(U)=\{C_i\}_{i\in I}$. Then we can assume that
 $U=\oplus_{i\in I}C_i^{a_i}$ and $V=\oplus_{i\in I}C_i^{b_i}$, where $a_i, b_i\in\mathbb Z_{\geq 1}$. 

Let $k\in I$. We know that $\mathcal M':=\oplus_{j=1}^n\mathcal M_j'=\mu_k(\mathcal M)$ is a left mutation. By Lemma \ref{lem:multi-brick} (ii), we have
\[ F_{C_k}[{\bf g}_{{\mathcal M}_k'}]\neq 0 \quad \text{and}
\quad F_{C_i}[{\bf g}_{{\mathcal M}_k'}]=0 \quad \forall i\in I\setminus\{k\}.
\]
Since  $U=\oplus_{i\in I}C_i^{a_i}$ and $V=\oplus_{i\in I}C_i^{b_i}$, we have $F_U=\prod_{i\in I}F_{C_i}^{a_i}$ and $F_V=\prod_{i\in I}F_{C_i}^{b_i}$. Then by Corollary \ref{cor-F-prod}, we have
\begin{eqnarray*}
   F_U[{\bf g}_{{\mathcal M}_k'}] =\sum_{i\in I}a_iF_{C_i}[{\bf g}_{{\mathcal M}_k'}]=a_kF_{C_k}[{\bf g}_{{\mathcal M}_k'}]\quad\text{and}\quad
    F_V[{\bf g}_{{\mathcal M}_k'}] &=&\sum_{i\in I}b_iF_{C_i}[{\bf g}_{{\mathcal M}_k'}]=b_kF_{C_k}[{\bf g}_{{\mathcal M}_k'}].
\end{eqnarray*}
Since $U$ and $V$ have the same Newton polytope, we have 
\[a_kF_{C_k}[{\bf g}_{{\mathcal M}_k'}]=F_U[{\bf g}_{{\mathcal M}_k'}]=F_V[{\bf g}_{{\mathcal M}_k'}]  =b_kF_{C_k}[{\bf g}_{{\mathcal M}_k'}].\]
Since $F_{C_k}[{\bf g}_{{\mathcal M}_k'}]\neq 0$, we obtain $a_k=b_k$. As $k$ is an arbitrary element in $I$, we have $a_i=b_i$ for any $i\in I$. Thus $U\cong V$.
\end{proof}

\bibliographystyle{alpha}
\bibliography{myref}

\end{document}